\documentclass[12pt]{amsart}

\usepackage[english]{babel}
\usepackage[utf8]{inputenc}
\usepackage[T1]{fontenc}

\usepackage[margin=1.5in]{geometry}
\usepackage[alphabetic]{amsrefs}

\usepackage[makeroom]{cancel}
\usepackage{setspace}
\usepackage{amsfonts} 
\usepackage{amssymb}
\usepackage{amsthm}
\usepackage{indentfirst}
\usepackage{enumerate}
 \usepackage{enumitem}   
\usepackage{amsmath}
\usepackage{graphicx}

\usepackage{mathabx}
\graphicspath{ {./images/} }

\usepackage{enumitem}

\usepackage{tikz-cd}

\newcommand{\bC}{\mathbb{C}}
\newcommand{\bH}{\mathbb{H}}

\newcommand{\bR}{\mathbb{R}}

\newcommand{\bF}{\mathbb{F}}
\newcommand{\bN}{\mathbb{N}}
\newcommand{\bT}{\mathbb{T}}

\newcommand{\cM}{\mathcal{M}}

\newcommand{\cS}{\mathcal{S}}

\newcommand*{\tran}{\mathsf{T}}
\newcommand*{\re}{\text{Re}}
\newcommand*{\Rea}{{\mathfrak R}}
\newcommand*{\im}{\text{Im }}
\newcommand*{\ncS}{\mathrm{ncS}}
\newcommand*{\ncconv}{\mathrm{ncconv}}
\newcommand*{\Epi}{\text{Epi}}

\newtheorem{theorem}{Theorem}[section]
\newtheorem{lemma}[theorem]{Lemma}
\newtheorem{prop}[theorem]{Proposition}

\newtheorem{cor}[theorem]{Corollary}
\newtheorem{thm}[theorem]{Theorem}

\newtheorem*{cor*}{Corollary}
\newtheorem*{thm*}{Theorem}
\newtheorem*{lem*}{Lemma}

\newtheorem*{prop*}{Proposition}

\theoremstyle{definition}
\newtheorem{definition}[theorem]{Definition}

\newtheorem{example}[theorem]{Example}

\newtheorem*{defn*}{Definition}

\theoremstyle{remark}

\title[Real Noncommutative Convexity II]{Real Noncommutative Convexity II:
Extremality and nc convex functions }

\author{David P. Blecher}
	
\address{Department of Mathematics, University of Houston, Houston, TX 77204-3008.}
	
\email{dpbleche@central.uh.edu}

\author{Caleb Becker McClure}

\address{Department of Mathematics, University of Houston, Houston, TX 77204-3008.}
	
\email{cbmcclur@central.uh.edu}
\subjclass[2020]{Primary: 46A55, 46L07, 47A20, 47L07, 26B25; Secondary 46L51, 46L52, 47B92, 47L05, 47L25, 47L30}

\begin{document}

\begin{abstract}  We continue the development of real noncommutative (nc) convexity, building on the recent and profound complex theory of Davidson and Kennedy. The present paper focuses on the theory of nc extreme points (and pure and maximal points) and the nc Choquet boundary in the real setting, as well as on the theory of real nc convex and semicontinuous functions and real nc convex envelopes. Our main emphasis is on how these notions interact with complexification.  In particular, parts of the paper analyze in detail how various notions of “extreme” or “maximal” relate to our earlier concept of the complexification of a convex set. Several new features emerge in the real case, especially in the later sections, including the novel notions of the complexification of a nc convex function and of the complexification of the convex envelope of a nc function.  
With an appendix by T.\ Russell.  
   \end{abstract}
\maketitle
\begin{section}{Introduction}

Very recently, in \cite{BMc}, we developed the real case of the theory of noncommutative (nc) convex sets in the sense of Davidson and Kennedy. In the present work, we continue the development of real nc convexity, guided by the recent and profound complex theory developed by those authors in the memoir \cite{DK}. As before, one of our primary goals is to understand the role and properties of complexification in this setting.
Several motivations for our work are described in \cite{BMc}. One, discussed there in more detail, is that classical convexity is, in many respects, essentially a real theory. Consequently, the real variant of nc convexity is likely to play an important role in future developments in operator algebras, functional analysis, mathematical physics, and quantum analysis
 (see for example \cite{CN} as a representative  example  of what part of this might look like).  The present paper focuses on the real analogues of topics corresponding to the middle sections of the memoir \cite{DK}.  This includes the theory of nc extreme points (and pure and maximal points), Arveson’s nc Choquet boundary in the real case, and the theory of real nc convex and semicontinuous functions and real nc convex envelopes.  In particular we emphasize how all of these interact with complexification. One use of this is as a tool for transferring results between the real and complex theories of nc convexity, often allowing one to circumvent technical difficulties. 
Portions of the paper analyze in detail how various notions of “extreme” interact with our earlier notion of the complexification of a convex set. More sophisticated developments, such as deeper aspects of nc Choquet theory, will be addressed in future work.

Several new features arise in the real case, particularly in the later sections of the paper. These include the novel notions of the complexification of a nc convex function on a nc convex set $K$, and of the complexification of the convex envelope of a nc function on $K$. As  just one example, we show that the convex envelope construction commutes with complexification. The convex envelope of a function $f$ defined on a convex set is of central importance in classical convexity theory, where it is sometimes referred to as the “lower envelope” (see, for example, \cite{LMNS}; the concave or upper envelope may of course be treated via the convex envelope of $-f$.
More generally, we continue to develop foundational structural results for real nc convex sets and the nc functions defined on them.

The real $C^*$-envelope of a real operator system, a.k.a.\ the {\em noncommutative Shilov boundary}, and most of its basic properties, were investigated some time ago (see e.g.\ the references at the start of \cite[Section 5]{BR}).  In early sections of our paper we investigate its construction via the more recent `extremal' and `nc Choquet boundary' approach.  We will see that the complexification of the nc Choquet boundary is incomparably  more complicated than the complexification of the nc Shilov boundary, or of maximal elements or maximal ucp maps. That is, there are significant complications in the relationship between complexification and pure or extreme points of nc convex sets, boundary representations, and related notions. This arises because the complexification of a nc extreme point need not remain nc extreme. That such behavior can occur is not surprising in light of analogous results in the recent paper \cite{RSS}, which treats the different  setting of {\em $C^*-$extreme maps} on real $C^*$-algebras.
By contrast, maximal elements—or maximal unital completely positive maps—behave perfectly under complexification, and these are sufficient to recover the noncommutative Shilov boundary. From the outset, it appeared plausible that the real extremal theory might differ substantially from the complex case, given its close connection to irreducible representations and the fact that irreducible representations of real $C^*$-algebras differ in important ways from their complex counterparts \cite{Li}. Fortunately, we find that no significant difficulties arise in this regard.

  There is of course a large and highly impressive literature (going back thirty years or so)  on noncommutative or quantum versions of extreme points, Krein-Milman, etc, and also on noncommutative versions of convex functions.  As just a few recent examples see for instance the remarkable papers \cite{EHKM, EPS, PaT, Pa} (and the list of papers referred to on p.\ 2 of \cite{BMc}),
and references therein.  (At the time of ArXiV submission of the present paper Scott McCullough also pointed out the very recent
\cite{EP}.) 
Indeed there are even classical precursors, e.g.\ in the operator (or matrix) convex and monotone functions going back to L\"owner.  Moreover this large existing theory of
`operator/matrix convexity' of which we speak,  often focuses on real examples  (e.g.\ spectrahedra), and sometimes is even
completely cast in the real framework, using `real' tools.
 However we believe that there is no technical overlap of our work with this literature, since our goal here is mainly to
 extend the perspective of \cite{DK} to the real case, particularly emphasizing and developing the role of the complexification in this endeavor.
 (In some stretches of Section 2 where the complexification plays little
part we are essentially just verifying and resolving potential complications in,
and otherwise expositing, the real case of [14, Section 6]. Our contributions
in this part of Section 2 are not intended to be novel, but rather to be systematic and useful
 for future ‘real users’.)    Indeed aspects of the  work earlier cited in this paragraph will hopefully interact  in the future very profitably with the matter of the present paper(s), and in any case this cited work  certainly furnishes very many interesting and important real examples.
Every complex nc convex set or complex nc convex function is a real one, but there are many interesting and important 
real examples which are not complex 
ones. 
Also, much of the work cited earlier in this paragraph  is stated in the Wittstock-Effros-Webster-Winkler language \cite{WW} of {\em matrix convexity}.
We refer readers to \cite{DK} for discussion of some of the differences between nc and matrix convexity.

Turning to the structure of our paper, in Section
        \ref{rnc} we begin by characterizing maximal elements
        in a complexification.  This gives a new route to the construction of the $C^*$-envelope of a real operator system (or unital operator space). We also discuss the real case of the theory of Davidson and Kennedy's nc convexity variant of 
        extreme, pure, and maximal points (which where `irreducible' are essentially Arveson's boundary representations).  And of  the associated nc Choquet and Shilov boundary, and its application to the nc Krein-Milman theorem and its converse.          In most of the rest of Section \ref{rnc} we give a few points of clarification about the real case of the arguments, 
        and note  a few places where the arguments are different in the real case.  
        In Section \ref{crco} we
        examine the connections between 
        nc extreme, pure, and maximal points in the real and complex case, that is, how the 
        these notions interact with the complexification of a nc convex set.  One may view some of our work in  this section as complementary  to parts of the very recent paper \cite{RSS}.  In that paper  some analogous  
        results are established for  $C^*$-extreme maps on real $C^*$-algebras.  Every pure  (and in particular any nc-extreme) ucp map  on a $C^*$-algebra is $C^*$-extreme, but not vice versa.  They do not treat general operator systems, thus their results would not apply to points in general nc convex sets. 
        
      In Sections \ref{ncf} and \ref{multi} we consider the real case of Davidson and Kennedy's nc convex functions, including their multivalued nc convex functions, and convex envelopes.   
                 We  give  many foundational structural results here, for example developing the theory of  the complexification in these settings and 
elucidating how it interacts with the basic objects and constructions.  E.g.\ we show  that the convex envelope construction commutes with complexification. As 
in 
\cite{BMc}, 
 the complexification then allows us to derive the real variant of profound and extremely technical results from      \cite[Section 7]{DK} very quickly from the complex case.  
 
 Finally, an Appendix by Travis Russell contains some details on the converse to the Krein-Milman theorem, in the real case.
    
 As usual, many of the results for real   operator systems            here or in \cite{BMc} apply or have variants for the larger class of
 real unital operator spaces: If $V$ is a
 real  unital operator space then $V + V^*, V \cap V^*$, and $\{ x + x^* : x \in V \}$ are well defined real   operator systems \cite{BR,BT}. 

      In keeping with the task and nature of our paper we will expect our readers to be a little familiar with, and 
      reading alongside with parts of, \cite{DK} and our prequels \cite{BMc,BR} (numbers in \cite{DK} refer to numbering in the v4 Arxiv version). 
 Because of this we  do not need to be very pedantic or overly careful with definitions, preliminaries, or the history of the subject, which may usually be found referenced in more detail there. 
        For general background on operator systems and spaces, and in particular on the definitions etc.\ in the rest of this section, we refer the reader to e.g.\ \cite{Pnbook,BLM,DK}, and in the real case to e.g.\ \cite{BReal,BR,BMc}.  
It might also  be helpful to also browse some of the other  existing real operator space theory  e.g.\ \cite{ROnr,RComp,
BT,BCK}.  Some basic 
real $C^*$- and von Neumann algebra theory may be found in \cite{Li}.  

We write $M_n(\bR)$ for the real $n \times n$ matrices, or sometimes simply $M_n$ when the context is clear.  Similarly in the complex case. 
  We write $\Rea \, a$ for $\frac{1}{2}(a + a^*)$, while for 
$z \in M_n(\bC)$ we write $\re \, z$ for $x \in M_n(\bR)$ where $z = x + iy$ for  $y \in M_n(\bR)$. 
We sometimes use the quaternions $\bH$ as an example: this is simultaneously  a real operator system, a real Hilbert space, and a real C$^*$-algebra, usually thought of as a real $*$-subalgebra of $M_4(\bR)$ or $M_2(\bC)$. Its complexification is $M_2(\bC)$.  
 The letters $H, L$ are usually reserved for real or complex Hilbert spaces.  Every complex Hilbert space $H$ is a real Hilbert space, i.e.\ we may forget the complex structure.   More generally we write $X_r$ 
 for a complex Banach space regarded as a real Banach space. 
 We write $X_{\rm sa}$ for the selfadjoint elements  
 in a $*$-vector space $X$.  In the complex case 
 $M_n(X)_{\rm sa} \cong (M_n)_{\rm sa} \otimes X_{\rm sa}$, but this  fails for real spaces. A subspace of $B(H)$ is {\em unital} if contains the identity,   
and a map $T$ is unital if $T(1) = 1$.  Our identities $1$ always have norm $1$.  For $z \in M_n(\bC)$ we write $\re \, z$ for $x \in M_n(\bR)$ where $z = x + iy$ for  $y \in M_n(\bR)$.  For a cardinal $n$ we define the isometry $u_n = \frac{1}{\sqrt{2}} \begin{bmatrix}
    1_n \\ -i \cdot 1_n
\end{bmatrix}$ where $1_n$ is the $n$-dimensional identity operator.  We sometimes often write this as $u$ (suppressing the subscript). 

The categories of real nc convex sets and real operator systems are equivalent  (see \cite[Section 4]{BMc}), 
and we often switch without warning between such sets $K$ and the associated system, e.g.\ viewing 
                $K = \ncS(V)$ for an appropriate real                 operator system $V$. 
 A {\em concrete complex (resp.\ real) operator system} $V$ is a unital selfadjoint subspace of $B(H)$ for $H$ a complex (resp.\ real) Hilbert space. For $n \in \bN$ we have the identification $M_n(B(H)) \cong B(H^{(n)})$ where $H^{(n)}$ is the $n$-fold direct sum of $H$. From this identification, $M_n(V)$ inherits a norm and positive cone. The latter is the set $M_n(V)^+ \coloneq \{x \in M_n(V): x = x^* \geq 0 $ in $ B(H^{(n)})\}$. For  $n \in \bN$ we define the amplification of a linear map $\varphi: V \to W$ by 
        \[\varphi^{(n)}: M_n(V) \to M_n(W)\]
        \[[x_{ij}] \mapsto [\varphi(x_{ij})] . \]
        The natural morphisms between operator systems are {\em  unital completely positive} (ucp) functions, which are linear maps $\varphi: V \to W$ that are unital and every amplification is positive (or equivalently selfadjoint and contractive). The isomorphisms (resp.\ embeddings) of operator systems which are used in this paper are bijective (resp.\ injective) ucp maps whose inverse (resp.\ inverse in its range) is ucp.  These are called unital complete order isomorphisms (resp.\ unital complete order embeddings).

        Similarly, a concrete operator space $E$ is a subspace of $B(H)$ with norms on $M_n(E)$ inherited from $B(H^{(n)})$. There are abstract characterizations of real/complex operator spaces and operator systems which we will not repeat here.
 If $E$ is a real operator system then $E$ can naturally be made into a complex operator system by complexification. The complexification $E_c$ of $E$ is the complex vector space $E_c$ consisting of elements $x+iy$ for $x,y \in E$. We give this a conjugate linear involution $(x+iy)^* = x^* - i y^*$. The matrix ordering $M_n(E_c)^+$ will be defined by
        \[M_n(E_c)^+ = \{x+iy \in M_n(E_c): c(x,y) \geq 0\}\]
        where $$c(x,y) = \begin{bmatrix}
            x & -y\\ y & x
        \end{bmatrix}  \in M_{2n}(E).$$ 
  We also sometimes write $c(x+iy)$ for $c(x,y)$.       We recall that this complexification  is the unique one satisfying Ruan's completely {\em reasonable} condition, namely that the map $\theta_E(x+iy)=x-iy$ is a unital complete order isomorphism.  
  Note that for a real $C^*$-algebra $A$, and in particular for $A = M_m(\bR)$ the map $c : A_c \to M_2(A)$ is a  faithful $*$-homomorphism. 

For a real or complex  operator space $E$ and a possibly infinite cardinal, $M_n(E)$ is the space of matrices whose `finitely supported' submatrices 
 have uniformly bounded norm.     In the case $E = \bF$ write $M_n = M_n(\bF)$. 
Thus  $M_n$ is $B(\ell^2_n)$.  Indeed for every Hilbert space $H$, $B(H) \cong M_n$ $*$-isomorphically for some $n$, 
 after one chooses an orthonormal basis. 
For a real operator space $E$, as in \cite{DK} we  define $\cM(E) = \bigsqcup_n \, M_n(E)$ (for $n$ cardinals bounded by some large enough 
cardinal $\kappa$) with $M_n(E)$ the matrix space of $E$ for $X \subseteq \cM(E)$.  Here $\bigsqcup_n$ is the disjoint union.
        Define $X_n = X \bigcap M_n(E)$. In the case $E = \bR$ write $\cM = \cM(\bR)$. 
          In \cite{BMc} we also defined  the novel notion of the complexification of a nc convex set.  
          (This has already found applications, for example in \cite{BEKMN}.)  
A real nc convex set $K$ is a subset of $\cM(E)$ for a real operator space $E$ that is graded, closed under direct sums, and closed under compressions. The set is compact if each level is compact in the point-weak$^*$ topology. This is the same definition as in the complex case except that direct sums and compressions are done using real matrices. Equivalently, it is a set closed under nc convex combinations. The complexification of $K$, denoted by $K_c$, is the subset of $\cM(E_c)$ 
 consisting of elements $x+iy$ such that $c(x,y) \in K$. 
Such expressions $c(x,y)$ or $c(x+iy)$ occur frequently in our paper(s).  Similarly, for a cardinal $n$ we define as in  \cite{BMc} the isometry $u_n = \frac{1}{\sqrt{2}} \begin{bmatrix} 
    1_n \\ -i \cdot 1_n 
\end{bmatrix}$ where $1_n$ is the $n$-dimensional identity operator.  We sometimes also write this as $u$. 
 In \cite{BMc} it is shown that 
$K_c$ is (the unique, under a very mild hypothesis) complex convex set containing $K$ such that $K$ generates $K_c$.  

        In parts of \cite{BR,BMc} it was checked that many of the basic theorems and constructions for complex operator systems also hold for real operator systems. Very many foundational
structural results for real operator systems were developed, and it was shown  how the complexification interacts
with the basic constructions in the subject.

        We write $V_r$ for a complex operator system regarded as a real one. By a (real) (resp.\ complex)  $C(X)$-space we will mean the continuous real (resp.\ complex)  functions on a compact Hausdorff space $X$.
        Real commutative $C^*$-algebras are not all of this kind, for example $\bC_r$ is not real $*$-isomorphic to a real $C(X)$-space.  Real $C(X)$ spaces may be viewed as the real commutative $C^*$-algebras which have identity (i.e.\ trivial) involution \cite{Good}. 
        
        \begin{section}{Extreme, pure, maximal points, and the nc Choquet boundary}
        \label{rnc}

Given a real nc convex set $K$, we define real maximal, pure and nc extreme points analogously to their complex counterpart.

            \begin{definition} \label{dkmax}
                For $x \in K_m$ and $y \in K_n$ we say that $x$ is a {\em compression} of $y$ and $y$ is a {\em dilation} of $x$ if there exists an isometry $\alpha \in M_{n,m}(\bR)$ such that $x = \alpha^\tran y \alpha$. A dilation $y$ of $x$ is trivial if $y$ decomposes with respect to the range of $\alpha$ as a direct sum of $x$ and some other element in $K$. Equivalently, $y$ commutes with the projection $\alpha \alpha^\tran$. A point $x \in K$ is {\em maximal} if every dilation of $x$ is trivial.
            \end{definition}

            \begin{definition}  \label{dkpure}
                A point $x \in K_n$ is called {\em pure} if when it is written as a finite nc convex combination, $x = \sum_i \alpha_i^\tran x_i \alpha_i$ for $x_i \in K_{n_i}$ and non-zero $\alpha_i \in M_{n_i, n}(\bR)$ such that $\sum_i \alpha_i^\tran  \alpha_i = 1_n$, then $\alpha_i = c_i \beta_i$ where $\beta_i$ is an isometry and $c_i$ is a positive number such that $\beta_i^\tran x_i \beta_i = x$.

                If   $K = \ncS(V)$ for an appropriate real  operator system $V$ as in \cite[Section 4]{BMc}, this corresponds to $\psi \leq \varphi$ for a cp map $\psi$, implying that $\psi$ is a positive scalar multiple of $\varphi$.  Indeed it is obvious
        that such a $\varphi$  is pure in the first sense above (note $x = t_i \alpha_i^\tran x_i \alpha_i$ evaluated at 1 implies $1 = t_i \alpha_i^\tran \alpha_i$, so $\alpha_i$ is a multiple of an isometry).  Conversely, if
        $\varphi = \varphi_1 +  \varphi_2$ for cp maps  $\varphi_i$, use \cite[Lemma 6.3]{BR} to write  
        $\varphi = b_1 \psi_1 b_1 +  b_2 \psi_2 b_2$ for ucp maps  $\psi_i$ and matrices $b_i \geq 0$ with $b_1^2 + b_2^2 = 1$.  If $\varphi$ is pure in the first sense then $\varphi_i = b_i \psi_i b_i$ is a positive multiple of $\varphi$.         This allows us in many situations (such as if the embedding of $K$ is `nc regular', or simply if 
        $K$ lies in a nc hyperplane not containing $0$, as in Remark 6.1.6 in \cite{DK}) 
        to drop the requirement that $\alpha_i$ is a positive scalar multiple of an isometry above, instead saying that 
           $\alpha_i^\tran x_i \alpha_i$ is a positive scalar multiple of $x$.             \end{definition}

            Combining the two definitions gives the real non-commutative analogue of extreme points.

            \begin{definition}  \label{dkext}
                A point $x \in K_n$ is called 
                (nc) {\em extreme} if when it is written as a finite nc convex combination, $x = \sum_i \alpha_i^\tran x_i \alpha_i$ for $x_i \in K_{n_i}$ and non-zero $\alpha_i \in M_{n_i, n}(\bR)$ such that $\sum_i \alpha_i^\tran  \alpha_i = 1_n$, then $\alpha_i = c_i \beta_i$ where $\beta_i$ is an isometry and $c_{i}$ is a positive number such that $\beta_i^\tran x_i \beta_i = x$ and $x_i$ 
                commutes with $\beta_i \beta_i^\tran$ for each $i$ (so that $x_i$ decomposes with respect to the range of $\alpha_i$ as $x_i = y_i \oplus z_i$ with $y_i$ unitarily equivalent to $x$). The set of real extreme points is denoted by $\partial K$ or Ch$(K)$.
            \end{definition}

 We sometimes write the collection of nc extreme points as   Ch$(K)$, since it corresponds to the Choquet boundary, not the Shilov boundary.  The later Corollary \ref{DK621} will say that these are 
 essentially the same as the Arveson {\em boundary 
 representations}.

Suppose that $\varphi : V \to B(H)$ is a unital complete order embedding on a real operator system $V$.  
The real $C^*$-algebra generated by the image of a maximal dilation $j : V \to B(L)$ of $\varphi$ is (a copy of) the important $C^*$-{\em envelope} $C^*_{\rm min}(V)$, by the  well known arguments for this from \cite{DM,ArvNote}.  That is, it has the universal property that for any unital complete order embedding $\kappa : V \to B(L)$ there is a unique 
$*$-homomorphism $\pi : C^*(\kappa(V)) \to C^*(j(V))$
with $\pi \circ \kappa = j$. The arguments (due to Dritschel-McCullough, Arveson,  and others, see e.g.\ the historical notes to \cite[Section 4.1]{BLM}) for the next result  work the same in both the real and complex case.  As explained and proved in several  places, e.g.\ 
in \cite[Proposition 2.2]{ArvNote} or \cite[4.1.11-4.1.13]{BLM}, we have:

\begin{thm} \label{dmar}   A ucp map $u : V \to B(H)$ is
maximal if and only if it has the {\rm unique extension property} (UEP).  Namely for every (or for some) $C^*$-algebra $B$ generated by a unital completely order isomorphic copy of $V$,
there is a unique ucp extension $u : B \to B(H)$, and this extension is a $*$-homomorphism. 
If this holds then it holds for $B$ equal to  $C^*_{\rm min}(V)$ or $C(K) = C^*_{\rm max}(V)$, indeed any ucp extension $\tilde{u} : B \to B(H)$ of $u$ factors through $C^*_{\rm min}(V)$. 
\end{thm}

 \begin{lemma} \label{rhom}  If $\pi : A \to B(H)$ is a unital $*$-homomorphism on a real $C^*$-algebra, then 
 $\pi$ is an extreme point in ${\rm UCP}(A,B(H))$.            \end{lemma}
    
                \begin{proof}  This is a well known result due to St{\o}rmer for complex $C^*$-algebras, and the real case follows immediately.   If $\pi =  \frac{1}{2} (\varphi  + \psi)$ where $\varphi, \psi \in
                {\rm UCP}(A,B(H))$ then $\pi_c = \frac{1}{2} (\varphi_c + \psi_c)$.  Thus $\varphi_c = \pi_c$ and $\varphi = \pi$. (See also \cite[Corollary 3.10]{RSS}
                for a complementary result.)
\end{proof}

We next prove a useful characterization of the maximal elements of a complexification.  
 We recall the fact from the first three lines of  the last paragraph of  the proof of \cite[Theorem 3.4]{BMc} that an element $x \in K_{2n}$ of a real nc convex set $K$ is of the form $c(a,b)$ for $a,b\in K_n$ if and only if for the unitary matrix $W = \begin{bmatrix}
     0 & -1 \\ 1 & 0
 \end{bmatrix}$ we have $W^*xW = x$. In this case, $c(u^*xu) = x$. Real nc affine maps preserve such elements.  Indeed if $\omega$ is nc affine then $W^*\omega(x)W = \omega(W^*xW) = \omega(x)$ if $x = c(a,b)$.  Hence also $c(u^*\omega(x)u) = \omega(x)$.  We use this in the proof below. 

\begin{thm} \label{cmax} For $K$ a real compact nc convex set, an element   $x \in K_c$ is maximal in $K_c$
 if and only if   $c(x)$ is maximal in $K$.            \end{thm}
    
                \begin{proof}  Suppose that $x \in K_c$ is maximal in $K_c$.  We may view $x = f + i g \in (K_c)_n$ as a (complex) maximal ucp map  $\varphi : V_c \to M_n(\bC)$ for an appropriate real operator system $V$, with $f,g \in {\rm UCP}(V, M_n(\bR))$.                 Recall that 
$\ncS(V_c) \cong \ncS(V)_c$ and $V_c = A(K_c)$ as in \cite[Section 3]{BMc}.  Thus $\varphi$ has the UEP by Theorem \ref{dmar}.
We may identify $c(x)$ or $c(\varphi)$ with the map 
$\psi : V \to M_{2n}(\bR)$ with $$\psi = c_n \circ   \varphi_{|V} : V \to M_{2n}(\bR),$$ where we write $c : M_n(\bC) \to M_{2n}(\bR)$ as $c_n$. That is, $\psi(v) = c(f(v),g(v))$.

We write $C^*(V)$ for a $C^*$-algebra generated by a  completely order isomorphic copy of $V$. Then $C^*(V)_c$
is a $C^*$-algebra generated by a  completely order isomorphic copy of $V_c$.
Suppose that 
$\psi$ has a ucp extension  $\rho : C^*(V) \to M_{2n}(\bR)$.   
Set $$\tilde{\rho} = \frac{1}{2} (\rho  + W^* \rho  W) : C^*(V) \to M_{2n}(\bR), \; \; \; \; \;  \; \;  W = \begin{bmatrix}
       0  & - I \\
       I & 0
    \end{bmatrix}.$$ Using the facts before the theorem and that $\psi(v) = c(\varphi(v))$ we see that this is also a ucp extension  of $\psi$.  Then $\tilde{\rho}_c : 
    C^*(V)_c 
    \to M_{2n}(\bC)$ is ucp.
Thus $u^* \tilde{\rho}_c(\cdot) u : C^*(V)_c \to M_{n}(\bC)$ is ucp, and extends $\varphi$.
Indeed writing $c : M_n(\bC) \to M_{2n}(\bR)$ as $c_n$, we have $$u^* \tilde{\rho}_c(v) u = u^* \tilde{\rho}(v) u = u^* \psi(v) u = u^* c_n(\varphi(v)) u = \varphi(v), \qquad v \in V.$$
Thus $u^* \tilde{\rho}_c(\cdot) u =  \varphi$ on $V_c$ by $\bC$-linearity. 
                Since $\varphi$ has the UEP we deduce that $u^* \tilde{\rho}_c(\cdot) u$ is a $*$-homomorphism  extending $\varphi$.  Thus 
                $$c_n \circ (u^* \tilde{\rho}_c(\cdot) u)_{|C^*_{\rm min}(V)} = c_n \circ (u^* \tilde{\rho}(\cdot) u) =  \tilde{\rho}$$  is a $*$-homomorphism   
                extending $\psi$. The second equality follows from the facts before the theorem.
                              By Lemma \ref{rhom} we have  $\tilde{\rho} = \rho$, and this is a $*$-homomorphism   
                extending $\psi$.   Thus $\psi = c(\varphi)$ has the UEP, and $c(x)$ is maximal in $K$. 
                
                For the converse direction, suppose that $c(x) \in K_{2n}$ is maximal, and that $x = \beta^* y \beta \in (K_c)_n$. Taking $c$ of both sides gives $c(x) = c(\beta)^\tran c(y) c(\beta)$ and therefore $c(\beta) c(\beta)^\tran c(y) = c(y) c(\beta) c(\beta)^\tran$. Adjoining $u_n$ to both sides gives that $\beta \beta^*$ commutes with $y$. 
                \end{proof} 
                
  \begin{cor} \label{maxiff}
  For $K$ a real compact nc convex set, if
$x \in K$  
then $x$ is maximal if and only if $x+i0 \in K_c$ is complex maximal.
                 \end{cor}
                 
            \begin{proof}  A direct proof of this result is possible (e.g.\ this is similar to the proof of a special case of Theorem \ref{pem} in an earlier ArXiV draft).             Instead we deduce it swiftly from Theorem \ref{cmax}. By that result $x+i0 \in K_c$ is complex maximal if and only if
            $x \oplus x$ is real maximal in $K$.   However it is a pleasant exercise that 
            this is equivalent to $x$ being maximal.
            \end{proof}

            Thus maximal elements or maximal ucp maps behave perfectly in regard to complexification.   (Confirming this, we shall show at the end of the section that if $K$ a complex compact nc convex set and 
$x \in K$, 
then $x$ is real maximal  in $K$ if and only if $x$ is complex maximal  in $K$.  This is not true for extremal elements.)

One may obtain the following Dritschel-McCullough theorem
in the real case by essentially  the same proof as in the complex case,  e.g.\ as it is presented in Arveson's account \cite{ArvNote} (an ordinal argument involving direct limits of Hilbert spaces, which includes a contribution by N. Ozawa).  
This was sketched in a previous version of our paper, in the Appendix by Travis Russell. We removed this since
it is essentially identical to Arveson's account, and since the latter is now published in \cite{Dav}. 
Instead, or alternatively, and for its independent interest, we will deduce it instantaneously from the complex case by complexification, as follows:

\begin{thm} \label{rdm} Every element $x$ of a 
real compact nc convex set $K$ has a  maximal dilation in $K$.  Equivalently, every ucp map $\varphi: V \to B(H)$ can be dilated to a ucp map $\pi: V \to B(L)$ which is maximal/has the UEP.
\end{thm}

\begin{proof}  Let $x \in K_n$ for a 
real compact nc convex set.
Then $x + i0 \in (K_c)_n$, so by the complex case of the present theorem it has a complex maximal dilation 
$x = \alpha^* y \alpha$, for an isometry $\alpha \in M_{m,n}(\bC)$ and a complex maximal  element   $y \in M_m(K_c)$.
Then $$x \oplus x = c(x) = c(\alpha^* y \alpha) = c(\alpha)^\tran c(y) c(\alpha) .$$ 
If $v^\tran =  \begin{bmatrix}
           I & 0 
       \end{bmatrix}$, then $w = c(\alpha) v$ is an isometry, and $x =w^\tran c(y) w$. Since  
       $c(y)$  is maximal in $K$ by Theorem \ref{cmax}, we are done.    \end{proof} 

This gives a (new) route to the construction of the $C^*$-envelope of a real operator system (or unital operator space).
Namely the real case of the Dritschel-McCullough-Arveson route (see \cite{DM,ArvNote} or \cite[Theorem 14.9.12]{Dav}, or as mentioned before
Theorem  \ref{dmar} or after Theorem  \ref{DK624}),  but using the  `real arguments' above. 
For example, if we take a  maximal dilation $\rho$ of the inclusion $V \subset B(H)$, using our swift 
Theorem \ref{rdm}, 
then the $C^*$-algebra generated by $\rho(V)$ is the real $C^*$-envelope.

 As in \cite{DK} we define  the {\em barycenter} of a ucp map $\mu : C(K) \to M_n$ to be $x \in K$ if $\mu_{|A(K)}$ equals evaluation at $x$.  We say $x$ has a {\em unique representing map} if $\delta_x$ is the unique such ucp $\mu$.   
 Since  $C(K) = C^*_{\max}(A(K))$,  (the real case of) Propositions 5.2.3 and 5.2.4 in \cite{DK} are also a rephrasing in the language above of parts of 
Theorem \ref{dmar}, and would need little thought  for those familiar with the basics here (and the proofs in \cite{DK} work in the real case). Thus:

\begin{cor} \label{sofoll} A point in a compact real nc convex set $K$ has a unique representing map if and only if it is maximal. \end{cor}

We say that $x \in K_n$ is {\em irreducible}
 if its (selfadjoint) image in $M_n$ (that is, $\{ f(x) : f \in A(K) \}$)  is irreducible in the sense that every element in $(M_n)_{\rm sa}$  commuting with $x$ 
 is  a  scalar multiple of $I_n$.
This is equivalent to $\delta_x$ being an  irreducible $*$-representation. 
Every pure $x \in K$ is irreducible, as in the complex case \cite{Arv}. 
If $K$ is a real nc convex set we say that $x \in K_n$ is {\em complex-irreducible} if its (selfadjoint) image  $\{ f(x) : f \in A_{\bR}(K) \}$ is irreducible as a subset of 
 $M_n(\bC)$.

As shown in Examples such as  \ref{Cal} below, a point $x \in K$ for real nc convex $K$ may be pure but $x \in K_c$ may not be. Similarly, by the statement after Lemma \ref{pekkcc}, a point $x \in K_c$ may be extreme while $c(x)$ is not.  Thus we cannot prove the Davidson-Kennedy theorem (see Theorem 6.2.2 in \cite{DK} or 14.10.8 in \cite{Dav})  via complexification, nor can we use complexification to deduce the existence of extreme points. Therefore, because pure (i.e.\ irreducible) points do not complexify well, we cannot do a proof of the real nc Krein-Milman theorem analogous to how we did Theorem \ref{rdm}. Thus, we now spend time to prove the real nc Krein-Milman theorem directly.

The discussion in 5.1.2 in \cite{DK} is the same in the real case.  Turning to Section 6 of \cite{DK},  Remark 6.1.3 in \cite{DK} and the following propositions have the same proof as in the complex case (see Proposition $6.1.4$ and $6.1.5$ in \cite{DK}).

            \begin{prop} \label{dk614} 
                For $K$ a real nc convex set, $x \in K$ is nc extreme if and only if it is pure and maximal.
            \end{prop}

            \begin{prop}\label{dk615} 
                For $K,L$ real nc convex sets and $\theta: K \to L$ an affine nc homeomorphism, $\theta$ maps maximal (resp. pure or extreme) points in $K$ to maximal (resp. pure or extreme) points in $L$.
            \end{prop}

 Item (1) in the following is also essentially mentioned in \cite{RSS}.

\begin{prop} \label{DK617} Let $\varphi : A \to M_n$  be a  ucp map on a real $C^*$-algebra   $A$.
\begin{enumerate} 
            \item [{\rm (1)}] $\varphi$ is pure if and only if it is a compression of an irreducible representation of $A$. If $A$  is a real $C(X)$ space  this forces the map to be a state.
 \item [{\rm (2)}] $\varphi$ is maximal (resp.\ nc extreme)  if and only if it is a $*$-representation (resp.\  irreducible  $*$-representation) of $A$. 
 \end{enumerate} 
\end{prop}

\begin{proof}  (1)  This is essentially \cite[Corollary 1.4.3]{Arv}.
The proof in the real case requires a couple of  additional observations.  E.g.\ we need to add to Arveson's proof the fact about real irreps from \cite[Proposition 5.3.7 (1)]{Li}.  Indeed the proof of \cite[Proposition 5.3.7 (1) and Proposition 4.3.4 (3) and (4)]{Li} shows that  $\varphi(V)'$ is a real von Neumann algebra whose projections norm densely span its selfadjoint part.  So $\varphi(V)$ is irreducible if and only there are no nontrivial such projections, or equivalently if $(\varphi(V)')_{\rm sa} = \bR \, I$.
        We also note that in Arveson's proof(s) in \cite[Section 1.4]{Arv} we may assume that the operators $T_i$ there are selfadjoint, and where he says that $\pi(B)'$ is scalar, we use $(\pi(B)')_{\rm sa}$ is scalar.
        For ucp maps Arveson's usage of `pure' in \cite{Arv} is equivalent to  Definition \ref{dkpure}. 
        
        (2)\ Follows as in Example 6.1.8  in \cite{DK}. \end{proof} 

See \cite{RSS} for several interesting results on pure and $C^*$-extreme maps in the real case.  E.g.\ Corollary 4.4 there extends the last assertion of (1) above.

The proof of the following extension of Corollary \ref{sofoll} is the same in the real case (see \cite[Theorem 6.1.9]{DK}).

\begin{theorem}\label{DK619}  
    For $K$ compact real nc convex, a point $x \in K_n$ is nc extreme if and only if the representation $\delta_x: C(K) \to M_n$ is both irreducible and is the unique representing map of $x$.
\end{theorem}

As
in \cite{DK} the next result 
follows directly from the last theorem. 

Arveson's {\em boundary representations} of an operator system are (in both real and complex case) the 
irreducible $*$-representations on a Hilbert space $H$ of a $C^*$-algebra $B$ generated by a unital completely order isomorphic copy of $V$, which extend a maximal ucp $\varphi : V \to B(H)$. 

\begin{cor}\label{DK621}
    For a real operator system $V$ with $K = \ncS(V)$, the nc extreme points of $K$ are precisely the restrictions to $V$ of boundary representations of $V$.  
\end{cor}

Thus maximal irreducible ucp maps $V \to B(H)$,
        or equivalently maximal irreducible elements in $K$, are essentially the `same thing' as Arveson's boundary representations. 

   For a $C^*$-algebra $B$ generated by a unital completely order isomorphic copy of $V$ we recall that the {\em Shilov boundary ideal} is largest closed ideal $J$ of $B$ such that the 
   quotient map $B \to B/J$ is completely isometric on $V$.  (The quotient by the Shilov boundary ideal is a copy of the $C^*$-envelope.) Arveson showed in the complex case \cite[Section 3]{Arv} that the Shilov boundary ideal equals  
   the intersection of the kernels of all boundary representations.  The same proof works in the real case.  His proof uses an `invariance principle' which was simplified in 
   \cite[Proposition 3.1]{ArvNote}.  It also follows from e.g.\ the observation above \cite[Proposition 4.1.12]{BLM}, which is the same in the real case. 
   
    The proof of the following result from \cite{DK0} is also  the same in the real case. 

\begin{thm} \label{DK024} Let $V$ be a real operator system and $K$ its associated real nc convex set.
    Every real pure ucp map $\varphi: V \to B(H)$ can be dilated to a nc extreme ucp map $\widetilde{\varphi}: \cS \to B(L)$.  Every real pure in $K$ has an extreme dilation. 
    \end{thm}

    \begin{thm} \label{DK025} Let $V$ be a real operator system and $K$ its associated real nc convex set.
    If  $x \in M_n(V)$ for $n < \infty$ then there is an $m$ and a nc extreme ucp map $\varphi: V \to M_m(\bR)$ with $\| \varphi_n (x ) \| = \| x \|$. For each $\epsilon > 0$ there is also a pure matrix state $\varphi: V \to M_k(\bR)$ for $k < \infty$ with $\| \varphi_n (x ) \| > \| x \| - \epsilon$. \end{thm}

\begin{proof} This is a combination of the real cases of Theorems 3.1, 3.3 and 3.4 in \cite{DK0}.  
All of this (including the alternative arguments given there) works the same in the real case, with as usual $\bC$ replaced by $\bR$, and recalling that all completely positive maps are selfadjoint. One argument 
there uses 
\cite[Theorem 2.2]{Far}, which states that for a real operator system $V$, finite $k$, and pure  cp map 
        $\psi : V \to M_k(\bR)$ there exists $m \leq k$, pure $\varphi \in {\rm UCP}(V,M_m(\bR))$ and $\gamma \in M_{m,k}(\bR)$ with $\psi = \gamma^\tran \varphi(\cdot)  \gamma$. 
In turn this uses in the second paragraph of its proof some complex finite dimensional linear algebra, and some tricks that 
        often fail for real spaces.   Because of its independent interest we check this 
        in the real case.    Let $b = \psi(1)$ and $p$ the projection onto Ran $b$.  
        To prove in that  second paragraph that $p \in \psi(V)'$ 
        we proceed as follows.  Consider completely positive  $\psi_c : V_c \to M_k(\bC)$, and view $p,b \in   M_k(\bR) \subset M_k(\bC)$.
        Indeed $p$ is then the projection in $M_k(\bC)$ onto $\bC^k$.  The proof in  that  second paragraph shows that
        $p \in \psi_c(V_c)'$, so that $p \in \psi(V)'$.
        \end{proof}

        We also of course now immediately obtain \cite[Theorem 6.2.4]{DK} in the real case. 
        
        \begin{theorem}\label{DK624}
            For $K$ a compact real nc convex set, define the Hilbert space $H = \oplus_n \oplus_{x \in (\partial K)_n} \, H_n$ and $\pi: C(K) \to B(H)$ the representation $\pi = \oplus_{x \in \partial K} \delta_x$. Then, the restriction $\pi|_{A(K)}$ is a unital complete order isomorphism   and $\pi(C(K))$ is $*$-isomorphic to $C^*_{\rm min}(A(K))$.
        \end{theorem}

Just as in the complex case, the $C^*$-algebra generated by the image of $V$ under a direct sum of sufficiently many extreme ucp maps (or boundary representations), is (a copy of) the $C^*$-envelope of $V$.
Indeed this direct sum  is maximal \cite[Corollary 3.5]{DK}. 
Thus, a quick method to construct the real $C^*$-envelope of a real  operator system $V$  in $B(H)$ is  to take a  maximal dilation $\rho$ of the inclusion $V \subset B(H)$.
The $C^*$-algebra generated by $\rho(V)$ is the real $C^*$-envelope.  (Note also that $\rho$ is maximal in $K_c$ by Corollary 
 \ref{maxiff}, so the $C^*$-algebra generated by $\rho_c(V_c) = \rho(V) + i \rho(V)$ is a copy of the complex $C^*$-envelope of $V_c$.)

            Example 6.1.10 in \cite{DK}  works essentially the same in the real case.  
            
         \begin{cor}  \label{DK6110} Let $C$ be a classical real compact convex set, and set $K = {\rm Min}(C)$ and $A = A(C)$, a unital real function space on $C$.   Then  $\partial K \subset K_1 = C$, and 
$\partial K$ is the classical Choquet boundary of $A$ in $C$, which is a subset of the classical Shilov boundary of $A$. \end{cor}

\begin{proof}  As in \cite[Section 6]{BMc}, 
$A(K)$ 
is an OMIN operator system unitally completely order isomorphic to OMIN$(A)$.   Clearly the $C^*$-envelope $C^*_{\rm min}(A(K))$
is a real commutative $C^*$-algebra.  However 
we need it to be the continuous functions on a
    compact set $X$, and this is not automatic in the real case.  Indeed this follows here because it is a quotient of the real valued continuous functions on $C$ (we are using Stone-Weierstrass to see that $A$ generates the latter; also the quotient of a $C(X)$ space is a $C(X)$ space).    As in the complex case explained there, or e.g.\ in 4.3.4 in \cite{BLM}, this implies that $X$ is the classical Shilov boundary of $A$, which is viewed as a closed subset of $C$.  To see that 
 $\partial K \subset K_1 = C$, and 
$\partial K$ is the classical Choquet boundary in $C$,  the argument in \cite{DK} works verbatim. 
For example, that if $x \in K_1$ is real extreme in $K$ then $x$ is an extreme point of the real state space of $A$, and hence is in the classical Choquet boundary.  It is well known in convexity theory that the basic classical  Choquet and Shilov boundary theory of a real function space works the same as the complex theory, except that signed measures are used rather than complex ones, via the real version of the Riesz representation theorem for measures. 
\end{proof} 

{\bf Remark.} General  real commutative $C^*$-algebras need not have the 
OMIN structure.

\begin{theorem}[Real noncommutative Krein-Milman]\label{DK642}
    For $K$ a real compact nc convex set,
    \[K = \overline{\ncconv}(\partial K) .\]
\end{theorem}

The proof is  the same as  in the complex case, following quickly from 
Theorem \ref{DK624} and \cite[Theorem 3.6]{BMc}.  See also the real matrix  Krein-Milman theorem in \cite{EPS}. 
The following is the real case of Davidson and Kennedy's noncommutative analogue of Milman's partial converse to the Krein-Milman theorem \cite[Theorem 6.4.3]{DK}:

\begin{theorem}\label{DK643}
    Let $K$ be a real compact nc convex set,  Let $X$ be a nc  closed subset of $K$, in the sense of that  every $X_n$ is closed in $K_n$ and whenever $x_i \in X_i$ is a net and $\{\alpha_i \in M_{n,n_i}(\bR) \}$ is a net of contractions such that $\lim \alpha_i \alpha_i^* = I_n$, and $\lim \alpha_i x_i \alpha_i^* = x$ exists, then $x \in X_n$.  Suppose that $X$ is also closed under compressions, i.e.\ $\alpha^* X_n \alpha \subset X_m$   for every isometry $\alpha \in M_{n,m}(\bR)$. If the closed nc convex hull of $X$ is $K$, then all real nc extreme points of $K$ are contained in $X$.   \end{theorem}

\begin{proof} For the real case, one needs to know that  
various properties from \cite{Dix} of irreducible representations and corresponding GNS representations used in \cite[Theorem 6.4.3]{DK}, as well as  Proposition 3.4.2 (ii) in \cite{Dix}, all hold in the real case.   The former properties are checked partially in \cite{Li} (see also \cite[Section 5]{BR}).  The latter result appeals to \cite[Lemma 3.4.1]{Dix}, which in turn relies on facts on extreme functionals from Proposition 2.5.5 there.  The latter is well known in the real case (see \cite[Proposition 5.2.7]{Li} and its proof).  Otherwise the proof proceeds almost exactly as in \cite[Theorem 6.4.3]{DK}.  For the readers convenience additional details are given in the Appendix below by T. Russell. 
\end{proof} 

\begin{thm}\label{DK651}
    Let $K$ be a compact real nc convex set and let $D$ be the noncommutative state space of $C^*_{\min}(A(K))$, identified with the set of noncommutative states on $C(K)$ which factor through $C^*_{\min}(A(K))$. Then $D$ is the closed nc convex hull of $\{\delta_x \in \partial K \}$. 
\end{thm}

The proof is  the same as  in the complex case.  We can identify $D$ as the {\em nc Shilov boundary}
of $V = A(K)$.

 \begin{example} \label{marh}    (See \cite{RSS} for an alternative take on, and proofs of,  some of the facts concerning $\bC_r$ below.) Note that  $\bH, \bC_r$ each have only one real $*$-irrep up to unitary equivalence. This follows from \cite[Proposition 5.3.7 (1)]{Li} because
       these algebras have only one real state, hence only one pure real state. Thus e.g.\ all irreps of $\bH$ are 4 dimensional and unitarily equivalent to the standard matrix form of $\bH$ in $M_4(\bR)$.   Thus the real nc Choquet boundary of $\bC_r$
       is the set of elements $U^T \, c \, U$ where $c : \bC \to M_2(\bR)$ is as usual, and $U$ is a rotation or reflection matrix.    On the other hand,
       the complexification of $\bC$ has two 
       irreducible representations, and the 
       complex nc Choquet boundary consists of two points.   This shows that in general for a general nc convex set $K$, there need be no nice relation between the 
       real Choquet boundary of $A(K)$ and the 
       complex Choquet boundary in $K_c$.  However the nc Shilov boundaries (which can be identified with the  $C^*$-envelopes) are of course closely related. 
       
       The noncommutative state space of $\bC_r$ is identifiable with the nc convex set  ${\mathcal C}{\mathcal S}$ of skew symmetric real 
       matrices of norm $\leq 1$.  One can show by functional calculus that these are exactly  the matrices
       $\varphi(i)$ for real ucp $\varphi : \bC \to M_n(\bR)$.  Indeed clearly such  $\varphi(i)$ is in ${\mathcal C}{\mathcal S}$.
       Conversely if $A \in {\mathcal C}{\mathcal S}$ then $iA \in M_n(\bC)_{\rm sa}$, so that 
       $iA = U^* {\rm diag}(d_1, \cdots , d_n) U$, where $d_k \in [-1,1]$ and $U$ is unitary.  The selfadjoint map taking $s+it \in \bC$ to $s I_n + tA = U^* {\rm diag}(s + i t d_1, \cdots , s + i t d_n) U$ in $M_n(\bR)$, is unital.  Here $s, t \in \bR$.
        It is also (completely) positive,  
       since the unital scalar functionals $s+it \mapsto s + i t d_k$ are ucp.   The latter may be seen in several ways.
       For example it is equivalent to $c(s,t) \mapsto c(s,t d)$ being ucp for $d \in [-1,1]$, which is clear for example since 
       multiplication by $d$ on $\bR$ induces a ucp map on the real Paulsen system of $\bR$. 
              
               The nc state space of $\bH$ was discussed earlier and in \cite{BMc}.
       These two nc  state spaces coincide with the  nc Shilov  boundary, and are the nc closed convex hull of the 
       nc Choquet boundary above, by Theorem 
       \ref{DK651}.  For the $\bC$ example this reveals 
       an interesting relation between the skew symmetric real 
       matrices of norm $\leq 1$, and the single matrix 
       $c(i) = \begin{bmatrix} 
           0 & -1 \\ 1 & 0 
       \end{bmatrix}.$  Namely, they are the noncommutative convex hull of $c(i)$: that is, ${\mathcal C}{\mathcal S} = {\rm nconv}(c(i))$. 
       
         All of the $C^*$-extreme points of ${\mathcal C}{\mathcal S}$ are computed in \cite{RSS}, and these turn out to 
       correspond to the $*$-representations.  So they are exactly the maximal elements by Proposition \ref{DK617}.
       The nc extreme points will be those $*$-representations which are irreducible, and we said at the start of this Example that these are
       unitarily equivalent to the usual representation $c : \bC_r \to M_2(\bR)$ (see also Example \ref{cise}). 
          The complexification of the nc convex set ${\mathcal C}{\mathcal S}$ turns out to be the operator interval $\{ x \in {\mathcal M}(\bC)_{\rm sa} : -I \leq x \leq I \}$.  See Example \ref{Cal} for more details.
       It would be quite interesting to 
       perform a similar analysis to the above with the quaternions.  \end{example}

       \begin{cor} For a real operator subsystem $V$ of $M_n(\bR)$ for $n < \infty$, the nc Choquet and Shilov boundaries coincide.  Indeed every irrep of $C^*_{\rm min}(V)$ is a boundary representation for $V$. 
\end{cor}

\begin{proof} In this case $C^*_{\rm min}(V)$ is a finite dimensional real $C^*$-algebra.
Hence by \cite[Theorem 5.7.1]{Li} we may assume that $C^*_{\rm min}(V) = \oplus_{k=1}^N \, M_{n_k}(D_k)$
where $D_k$ is $\bR, \bC$ or $\bH$.  The real irreps of this sum 
are up to unitary equivalence the obvious ones 
$\pi_k$ associated with the summands.  For example, 
corresponding to the canonical irreps 
$M_n(\bC_r) \to B(\bR^{2n})$ and $M_n(\bH) \to B(\bR^{4n})$.  This is easy to see using the fact that $\bH, \bC_r$ each have only one real $*$-irrep up to unitary equivalence.  We finish by an argument of Arveson: For any $j$, we have 
$\cap_{k \neq j} \, {\rm Ker}(\pi_k) \neq (0)$,
obviously.  However if $\pi_j$ is not a boundary representation then $\cap_{k \neq j} \, {\rm Ker}(\pi_k) = (0)$ since this collection includes (up to unitary equivalence) all boundary representations, and there are sufficiently many boundary representations.
\end{proof} 

 \begin{cor} The nc Choquet boundary of a general  real operator system $V$ is dense in the nc Shilov boundary (thought of as 
 the spectrum (equivalence classes of real irreps) of $C^*_{\rm min}(V)$).  Here by `dense' we mean in the  usual (non-Hausdorff Jacobson) topology on the spectrum. 
 \end{cor}

\begin{proof} The intersection in $C^*_{\rm min}(V)$ of the kernels of the boundary representations is $(0)$ (e.g.\ see the discussion above Theorem \ref{DK024}).  Hence the corollary  follows by the real case of basic facts about the topology of the spectrum of a $C^*$-algebra, such as \cite[Theorem 3.4.10]{Dix}. 
Some of these basic facts may be found in  \cite{Li} (e.g.\ Section 5.3 there), and in places in the proof of Theorem 5.3 in \cite{BR}), or have already been discussed in the proof of Theorem \ref{DK643} above.  Any other aspects of the real case of classical results about the spectrum $\hat{A}$ in early chapters of \cite{Dix} which we need here are easy to check and left to the interested reader.  \end{proof}

We end this section with a result continuing the earlier theme that  `maximal' elements behave perfectly with respect to ``real versus complex distinctions'' or complexification.

   \begin{thm} \label{pem} For $K$ a complex compact nc convex set and 
$x \in K$, 
then $x$ is complex maximal in $K$ if and only if it is real maximal in $K$.
\end{thm}

\begin{proof}   Certainly every complex maximal is real maximal.    For the other direction, we may suppose that $0 \in K_1$ by translation.
Suppose that  $x \in K_n$ is real maximal, and let $v =[I_n \; 0]^\tran$.  Write $x = \alpha^*y\alpha$ for $y \in K_{m}$ and $\alpha \in M_{m,n}(\bC)$ an isometry.   
We may write $\alpha = \beta + i \lambda$ with $\beta, \lambda \in M_{m,n}(\bR)$. 
Taking $c$ of both sides of this and compressing both sides by $u_n$ gives 
$x= u_n^* c(\alpha)^\tran c(y) c(\alpha) u_n$.   Let $w = c((1-i)/\sqrt{2})$, a projection having $u$ as eigenvector with eigenvalue $\frac{1}{\sqrt{2}}$.  It is easy to see that $c(\alpha) u_n = w_m c(\alpha) v$.   Let $\gamma = c(\alpha) v = [\beta \; \lambda]^\tran$, so that
$c(\alpha) u_n = w_m \gamma$.   We have 
$\gamma^\tran \gamma = \beta^\tran \beta + \lambda^\tran \lambda = I_n$, since   $$I_n = \alpha^*  \alpha ={\rm Re} \,  (\beta^\tran - i \lambda^\tran) (\beta + i \lambda) \, = \, \beta^\tran \beta + \lambda^\tran \lambda.$$
Let $z = w_m^* c(y) w_m = w_m^* c(y) w_m + q 0 q \in K$, where $q$ is the complementary projection to $w_n$. 
Then  $$x = u_n^* c(\alpha)^\tran c(y) c(\alpha) u_n =  \gamma^\tran w_m^*  c(y)  w_m \gamma = \gamma^\tran z \gamma.$$
Since $x$ is real maximal we get that  the projection $$P = \gamma \gamma^\tran =  c(\alpha) \begin{bmatrix}
                    1_n & 0 \\ 0 & 0
                \end{bmatrix}c(\alpha)^\tran$$ commutes with $z$.          Therefore, $P(w_m c(y) w_m)= (w_m c(y)w_m) P$ and adjoining $u_m$ to this gives
                \[u_m^*c(\alpha) \begin{bmatrix}
                    1_n & 0 \\ 0 & 0
                \end{bmatrix}c(\alpha)^\tran w_m c(y) w_m u_m = u_m^*w_m c(y) w_m c(\alpha) \begin{bmatrix}
                    1_n & 0 \\ 0 & 0
                \end{bmatrix}c(\alpha)^\tran u_m.\]
                As in \cite[Lemma 3.8]{BMc}   for $c(a)$ in either $M_{2m,2k}(\bC)$ or $M_{2m,2k}(K)$ we have that $u_m^* c(a) = au_k^*$ and similarly $c(a) u_k = u_m a$. 
                Thus $$u_m^* \, c(\alpha) = \alpha \, u_n^* \; , \; \; \; \; \; \; c(\alpha)^\tran \,  u = u \, \alpha^*.$$  Recall that $w u = \frac{1}{\sqrt{2}} u$, and that $c(y) = y \oplus y$ `commutes with' block matrices whose blocks are scalar multiples of
                $I_m$, so that $$w_m c(y) w_m u_m =  \frac{1}{\sqrt{2}}  \, w_m  c(y)  u_m =  \frac{1}{2} \,  u_m \, y.$$ Using these two displayed equations, the left hand side of the displayed equation above  those, turns into
                         \begin{align*}
                u_m^*c(\alpha) \begin{bmatrix}
                    1_n & 0 \\ 0 & 0
                \end{bmatrix}c(\alpha)^\tran w_m c(y) w_m u_m  & = \frac{1}{2} \, \alpha u_n^* \begin{bmatrix}
                    1_n & 0 \\ 0 & 0
                \end{bmatrix}c(\alpha)^\tran u_m y \\
                                    &= \frac{1}{2} \, \alpha u_n^*  \begin{bmatrix}
                    1_n & 0 \\ 0 & 0
                \end{bmatrix} u_n \alpha^*                  y \\
                                & =     \frac{1}{4} \, \alpha \begin{bmatrix} 
                    1_n & i 1_n 
                \end{bmatrix}
                \begin{bmatrix}
                    1_n & 0 \\ 0 & 0 
                \end{bmatrix}
                \begin{bmatrix}
                    1_n \\ -i 1_n 
                \end{bmatrix} \alpha^* y  
                \\&= \frac{1}{2} \, \alpha \alpha^* y, 
                \end{align*}
                and similarly the right hand side turns into $\frac{1}{2} y \alpha \alpha^*$ and therefore
                \[\alpha \alpha^*y = y \alpha \alpha^*.\] 
                So, $y$ decomposes with respect to the range of $\alpha$ as a direct sum. \end{proof}

 \section{Complexification of  extremals}    \label{crco}  In this section we
        examine the connections between 
        nc extreme, pure, and maximal points in the real and complex case, that is, how 
        these notions interact with the complexification of a nc convex set.

Classically pure (or equivalently, classical extreme \cite[Proposition 5.3.5 (1)]{Li}) elements 
 in $K_1$, are clearly just the nc pure states (in both the real and complex case). For example, if a state $\varphi$ is  nc pure and 
 is a convex combination of other states, then it is easy to see that these states must be $\varphi$. 
 However they are not necessarily nc  extreme/maximal. An example is the obvious real pure state on $\bC$, which is not maximal.

 We shall see later that a (scalar valued) state
 $\varphi$ on a real operator system $V$ is  nc extreme if and only if
 $\varphi_c$ is complex nc extreme on $V_c$. 
 Similarly we shall see that, in the language above, for every pure element  $x \in K_1$ there exists $y$ with $c(x,y) \in K_2$ and $x+iy$ pure in $K_c$ (which follows if  $c(x,y)$ is pure in $K_2$).

          \begin{lemma} \label{pekkc} For $K$ a real compact nc convex set, if
$x \in K$  is pure (resp.\ nc extreme) in $K_c$ then $x$  is pure (resp.\ nc extreme)
in $K$.    The converse is false:   pure (resp.\ nc extreme) 
in $K$ need not imply pure (resp.\ nc extreme)
in $K_c$.   In particular, in general
    $$K \cap \partial_{\bC} K_c  \subsetneq \partial_{\bR} K.$$ 
\end{lemma}

            \begin{proof}  Suppose that  $x$  is pure in $K_c$, and that $x =  \sum_k \alpha_k^\tran x_k \alpha_k \in K_n$ for finite families $\{x_i \in K_{n_i}\}$ and $\{\alpha_i \in M_{n_i,n}(\bR) \}$, where $\sum \alpha_i^\tran \alpha_i = 1_n$.  Then since 
 $x$  is pure in $K_c$ we have each $\alpha_i$ is a positive scalar multiple of a complex isometry matrix
 $\beta_k \in M_{n_k,n}(\bC)$ (in particular $n_k \geq n$),
 with $\beta_k^* x_k \beta_k = x$.  Clearly $\beta_k \in M_{n_k,n}(\bR)$.
 So we have verified that $x$  is pure in $K$. 
 Note further that $\beta_k \beta_k^* x_k  \beta_k \beta_k^* = \beta_k x \beta_k^\tran$.  

  If $x$ is nc extreme in $K_c$ then in addition each 
 $x_k$  commutes with $\alpha_k \, \alpha_k^\tran$. 
 
 The converses of these are false even if $K$ corresponds to 
 a commutative real $C^*$-algebra \cite{RSS}.  For example on $A = \{ f : \bT \to \bC : f(\bar{z}) = \overline{f(z)} \}$ the state 
 $f \mapsto {\rm Re} \, f(i)$ is classically pure and extreme
 in $S(A)$, thus is nc pure,
  but is not pure nor extreme in any reasonable sense in 
 $K_c$.   Also consider the restriction of the  trace on $M_2(\bC)$ to the copy of the quaternions. 
 See also Example  \ref{cise}. \end{proof}

  \begin{cor} \label{chekkc} For $K$ a real compact nc convex set, if  $x \in K_n$ is nc extreme in $K$ then  $x$  is  nc extreme in $K_c$
  if and only if $x$ is complex-irreducible in $M_n(\bC)$, or equivalently every nonzero positive in $M_n(\bC)$  commuting with $x$ is
  a positive scalar multiple of $I_n$.   
  In particular, $x \in K_1$ is nc extreme in $K$ if and only if $x + i0$ is nc extreme in $K_c$.
  \end{cor}
    
                \begin{proof}
                If $x$ is extreme  in $K$
 then $x$ is maximal in $K_c$.  Suppose that $x =  \sum_k \alpha_k^* x_k \alpha_k \in (K_c)_n$ for finite families $\{x_i \in (K_c)_{n_i}\}$ and $\{\alpha_i \in M_{n_i,n}(\bC)\}$, where $\sum \alpha_k^* \alpha_k = 1_n$. 
 If $\alpha = [\alpha_i]$ (a column) then  $\vec \alpha \vec \alpha^*$ commutes with diag$ \{ x_k \}$ (using $x$ is maximal in $K_c$).  That is, 
 $\alpha_i \alpha^*_j x_j = x_i \alpha_i \alpha^*_j$ for each $i,j$.  In particular this holds when $i=j$, which is one ingredient 
 of the extremality condition.  For the other ingredients, 
 letting $r_k = \alpha_k^* \alpha_k$ we have
 $r_i \alpha^*_j x_j \alpha_j =  \alpha^*_i x_i \alpha_i  r_j$.
 Summing over $i$ we have $\alpha^*_j x_j \alpha_j = x r_j$.
 Similarly this equals $r_j x$, so $\alpha^*_j x_j \alpha_j = r_j^{\frac{1}{2}} x r_j^{\frac{1}{2}}$ for each $j$.  So $x = \sum_k \, r_k^{\frac{1}{2}} x r_k^{\frac{1}{2}}$.

From the last paragraph we have reduced checking that  $x$  
 is  extreme in $K_c$ to the special case:   if $x = \sum_k \, a_k x a_k$
 for positive  $a_k \in M_n(\bC)$ with $\sum_k \, a_k^2 = 1$
 and $a_k x = x a_k$, then 
 $a_k$ is  a positive scalar multiple of $I_n$. 
 That is, for positive  $b_1, \cdots , b_m  \in M_n(\bC)$ with $\sum_k \, b_k = 1$
 and $b_k x = x b_k$, then $b_k$ is  a positive scalar multiple of $I_n$. 
 It is easy to see that in this last condition we may assume that $m = 2$.   
 If this holds with $m = 2$, and if 
 $b$  is positive in $M_n(\bC)$  and $b$ commutes with $x$, then $\| b \| I-b$
 satisfies the same conditions.  So $b$
 is a positive scalar multiple of $I_n$.   Thus  $x$  
 is  extreme in $K_c$   if and only if   every nonzero positive in $M_n(\bC)$  commuting with $x$ is
  a positive scalar multiple of $I_n$. 
 Now the result is clear.  Indeed the commutant in $M_n(\bC)$ of the range of $x$ is a
 von Neumann algebra, so is trivial if and only if it contains no
 nontrivial projections.  

This  criterion is true for $n = 1$ clearly, giving (together with the last lemma) the final assertion. \end{proof} 

{\bf Remarks.}  1)\ 
The class of complex-irreducible nc extreme or pure real matrix states is in general not big enough to norm an operator system or even a real $C^*$-algebra.
Indeed for  $\bC_r$ there are no complex-irreducible extreme real matrix states (as is said elsewhere the irreducible real matrix states  of $\bC_r$ are all unitarily equivalent to the one in Example \ref{cise}, and none of these are complex-irreducible).  

 \smallskip  
 
         2)\        It is worth remarking that 
               any complex compact nc convex set that is not a complexification, is probably in a sense not easy to describe.  Indeed such would give a 
               complex operator system that is not an  operator system complexification, and there are no 
                known simple examples of this at this date (we thank Chris Phillips for a recent conversation on this point).  For $C^*$-algebras this is equivalent to the hard problem of the existence of a real form (see e.g.\ first page or two of \cite{Ros}). 
                There are not even any known simple low dimensional examples of 
complex Banach spaces that are not  (isometric) complexifications.

\begin{prop} 
\label{finl}  
A complex ucp map $\varphi : V \to M_n(\bC)$ on a complex operator system is complex maximal in $K = \ncS_{\bC}(V)$ if $c_n \circ \varphi$
is real maximal in ${\rm UCP}_{\bR}(V,M_{2n}(\bR))$. 
Here  $c : M_n(\bC) \to M_{2n}(\bR)$ is written as $c_n$. \end{prop} 

\begin{proof} 
Suppose that $\varphi = \alpha^* \psi  \alpha$ for $\psi \in {\rm UCP}_{\bC}(V,M_m(\bC))$ and 
an isometry $\alpha \in 
M_{m,n}(\bC)$. View $\psi$ as in ${\rm UCP}_{\bR}(V,M_{2m}(\bR))$ and $c_n \circ \varphi$ in ${\rm UCP}_{\bR}(V,M_{2n}(\bR))$.
Then $c_n \circ \varphi = c(\alpha)^\tran \, (c_m \circ \psi) \, c(\alpha)$.   Since $c_n \circ \varphi$ is real 
maximal in $\ncS_{\bR}(V)$ we have that $c(\alpha ) c(\alpha )^\tran = c(\alpha \alpha^* + i0)$ commutes with $c_m \circ \psi$.  However this forces $\alpha \alpha^*$ to commute with $\psi$ as desired.  \end{proof}

If $x$ is pure in $K$ then it need not be real nor complex pure  in $K_c$ (for examples see \ref{cise} and \ref{Cal},
and Lemma \ref{pekkcr}).  
However if $x$ is nc extreme in $K$ then  it is real nc extreme (hence real pure) in $K_c$, 
by a small variant of the proof of Corollary \ref{chekkc}.

\begin{example} \label{cise}  We exhibit a simple example of a nc extreme $x \in K_2$ 
which is not nc extreme in $K_c$.   Indeed the criterion in Corollary \ref{chekkc} is not true in general for $n = 2$.
For $A = \bC$  consider the usual representation $c : A \to M_2(\bR)$.  This is a real irreducible $*$-representation, hence is nc extreme and pure 
by 
Proposition \ref{DK617} (2).
However
if we take 
$$b_1 = \frac{1}{2} \begin{bmatrix}
                    1 & -i \\ i & 1
                \end{bmatrix}, \; \; \;  \; \; \;
               b_2 = \frac{1}{2} \begin{bmatrix}
                    1 & -i \\ i & 1
                \end{bmatrix} ,$$ 
then $c$ commutes with the $b_i$, and $\sum_k \, b_k = 1$.
So the criterion in Corollary \ref{chekkc}  fails.  
\end{example}

 For the next two examples, we use the real compact nc convex set
  \[K = \ncconv\{c(i),c(-i)\} = \ncconv \left\{ \begin{bmatrix}
            0 & -1 \\ 1 & 0
        \end{bmatrix}  \right\} \subset \cM(\bR).\]
In Example \ref{marh} we observed that $K$ may be viewed as the nc state space of $\bC_r$,  and that $K$  is the skew matrix set ${\mathcal CS}$.  So $A(K) \cong \bC_r$ as operator systems.
           
    \begin{example}\label{rOpInterval}
     Let $I = \ncconv_\bR\{\pm1\}$. This equals the  compact real operator interval $\bigsqcup_n[-1,1]_n$. To see this, observe that $\subseteq$ is obvious and for the converse we observe that $ \ncconv_\bR\{\pm1\} =  \ncconv_\bR\{[-1,1]\}$. Elements of the compact operator interval can all be written as the diagonalization $U^\tran D U$ for $D$ the direct sum of elements in the interval $[-1,1]$. Thus, all elements in the operator interval are in $I$, which gives $\supseteq$.
    
     Elements of $A(I)$ are determined by where they send $\pm 1$. Therefore, $A(I)$ is a two dimensional real operator system, indeed it is a real $C^*$-algebra $*$-isomorphic to $\ell^\infty_2(\bR)$. The isomorphism is the map $\theta:A(I) \to \ell^\infty_2(\bR)$ taking $f \in A(I)$ to $(f(-1),f(1))$. Note that $A(I)$ has trivial involution, so all we need to check is that $\theta$ is multiplicative. However,
        \[\theta(fg) = ((fg)(-1), (fg)(1)) = (f(-1) \cdot g(-1), f(1) \cdot g(1)) = \theta(f) \cdot \theta(g).\]
        Because $\theta$ is bijective, this is a $*$-isomorphism.
Although both $I$ and ${\mathcal CS} = K$ have the same complexification they are not nc affine homeomorphic, indeed their associated operator systems are $\ell^\infty_2(\bR)$ and $\bC_r$, which are quite different.  Indeed although $I = {\rm Min}([-1,1])$  in the language of \cite[Section 6]{BMc},   $K$ is not a Min nc convex set.  Indeed $A(K) = \bC_r$ is not an real OMIN operator system: it has nontrivial involution and is not a real function system in the technical sense of \cite{BR,BMc}.
\end{example}

We now explore $K$ and its complexification $K_c$ in detail, and use it to construct another example of a pure element $x \in K$ such that $x+i0$ is not pure in $K_c$.
 
   \begin{example} \label{Cal}   Consider $K$ as defined above. From Theorem $3.1$ in  \cite{BMc}, we know that $K_c = \ncconv_{\bC}\{\ncconv_\bR\{c(i)\}\} = \ncconv_\bC\{c(i)\}$. A compression by $u$ shows $\pm i \in K_c$ so that $\ncconv_\bC\{\pm i\} \subseteq K_c$. Notice that we have
         $$c(i) =\frac{1}{2} \begin{bmatrix}
            1 \\i
        \end{bmatrix} i \begin{bmatrix}
            1 & -i
        \end{bmatrix} + \frac{1}{2}\begin{bmatrix}
            1 \\-i
        \end{bmatrix} (-i) \begin{bmatrix}
            1 & i
        \end{bmatrix}, $$
        so that $K_c = \ncconv_
        \bC\{\pm i\}$. Let $L = \ncconv_{\bC} \{\pm 1\}$,  the operator interval $\{ x \in {\mathcal M}(\bC)_{\rm sa} : -I \leq x \leq I \}$.   Because $\pm i \in i \cdot L$ we have that $K_c \subseteq iL$. Conversely, $\pm 1 \in -i \cdot \ncconv \{\pm i\}$ so that $L \subseteq -i K_c$.  This  implies $iL \subseteq K_c$ and thus $iL = K_c$. However, $L = \ncconv\{\pm 1\} = \bigsqcup_n [-1,1]_n$ by Example $4.1$ in \cite{WW} or as can be deduced from Example \ref{rOpInterval}. Also, $iL \cong L$ because multiplication by $i$ is a nc affine homeomorphism. Now, $L$ is compact, implying that $K_c$ is too, and thus $K$ is compact by Theorem 3.1. We have that $\pm1 \in L$ are the only nc extreme points also by Example $4.1$ in \cite{WW}, so that $\pm c(i)$ are the only possible extreme points in $K$, and they are indeed extreme. Also, $0 \in K_c$ is not pure as it is a convex combination of $\pm i$.

        Next we see that $0 \in K$ is pure: notice that for any matrix $v \in M_{2,1}(\bR)$ we have that $v^\tran c(i) v = 0$. If $0$ were written as a finite nc convex combination of points in $K$, say $0 = \sum_i \alpha_i^\tran x_i \alpha_i$ for $x_i \in K_{n_i}$ and non-zero $\alpha_i \in M_{n_i,1}(\bR)$, then each $\alpha_i$ is automatically a positive scalar multiple of an isometry so we just need to show $\alpha_i^\tran x_i \alpha_i = 0$. However, each $x_i$ can be written as a nc convex combination of $c(i)$, say $x_i = \sum_j \beta_{i,j}^\tran c(i)\beta_{i,j}$. Therefore, because $v^\tran c(i) v = 0$ for any row vector $v$, we get
        \[\alpha_i^\tran x_i \alpha_i = \sum_{j} (\beta_{i,j}\alpha_i)^\tran c(i) \beta_{i,j} \alpha_i = 0,\] and so $0$ is pure in $K$. It is worth noting that $0$ is not however maximal.
Similar arguments work for the example
        $K = \ncconv\{1, c(i) \}.$   
 
        \end{example}

Extreme in $K$ need not imply complex extreme
in $K_c$ as we just saw.  However, the following is a kind of a converse.
    
     \begin{lemma} \label{pekkcc} For $K$ a real compact nc convex set, if  $x \in K_c$ has $c(x)$ pure (resp.\ nc extreme, maximal) in $K$ then $x$  is pure (resp.\ nc extreme, maximal) in $K_c$.            \end{lemma}
    
                \begin{proof} The maximal case was already proved in Theorem \ref{cmax}. Suppose that $x =  \sum_k \alpha_k^* x_k \alpha_k \in (K_c)_n$ for finite families $\{x_i \in (K_c)_{n_i}\}$ and $\{\alpha_i \in M_{n_i,n}(\bC)\}$, where $\sum \alpha_k^* \alpha_k = 1_n$.  Then 
                $$c(x) =  c(\sum_k \, \alpha_k^* x_k \alpha_k) 
                =  \sum_k \, c(\alpha_k)^\tran c(x_k) c(\alpha_k) \in K_{2n},$$ and $\sum \, c(\alpha_k)^\tran c(\alpha_k) = 1_{2n}.$ 
                Since 
     $c(x)$  is pure in $K$, we have that a positive scalar multiple of $c(\alpha_k)$ is a  real isometry matrix
     $\eta_k \in M_{2n_k,2n}(\bR)$,
     with $\eta_k^\tran c(x_k) \eta_k = c(x)$.  
     It is then easy to see (since $c$ is one-to-one and `modular') that a positive scalar multiple
     of $\alpha_k$ is a complex isometry $\beta_k \in M_{n_k,n}(\bC)$,
     with $\beta_k^* x_k\beta_k = x$. So $x$  is pure in $K_c$.      The rest follows since `extreme' is `pure and maximal'. 
     \end{proof}

         The converse of the lemma is false even in the simplest example of $K$ singleton: $1$ is pure and 
         extreme in $\ncS(\bC)$ but 
    $c(1) = 1 \oplus 1$ is not pure, and is clearly not nc extreme.  The class of extreme points considered in this lemma is interesting in some examples, but not in others.

    \bigskip

 {\bf Remarks.} (On real parts of complex matrix states.) 
 1)\ By the argument in \cite[Proposition 5.3.3]{Li}, every pure state
 $\varphi$ on a real operator system $V$ is the restriction to $V$ of the real part of a pure state on $V_c$ (not usually $\varphi_c$). 
   However not every pure (resp.\ nc extreme) element  $x \in K$ is the real part of a pure (resp.\ nc extreme) element of $K_c$, sadly.  See \cite{RSS}, or 2) below, or Example \ref{cise}.  
    In the latter example, since $\ell^\infty_2$ has no irreducible representation into $M_n(\bC)$ for $n \geq 2$ we have no extreme points in level 2 (or higher) of the complexification.  Similarly for pure points: in that 
example there are no pure elements of $K_2 = {\rm UCP}(\bC,M_2(\bC))$ by 6.1.7 in \cite{DK}. 

The  classical argument that  pure/extreme states are the real part of a   pure state extension to the complexification  is a standard one in convexity (finding an extreme point of the set of extensions and showing that this is extreme in the superspace).   That the  nc extension of this fails illustrates the failure in the nc situation of this very common trick in convexity.
We do not know  a characterization of when a nc pure (resp.\ nc extreme) in K is the real part of some  pure (resp.\ nc  extreme)  in $K_c.$

2)\ (A quaternionic counterexample.)\ Similarly the identity map on $M_2(\bC)$ is nc extreme, so maximal.  However its real part $\psi$, which is Re,  restricted to the copy of the quaternions $\bH$, is not real maximal, so not extreme.  This can be seen from  Proposition \ref{DK617} (2), or by the following direct argument: $\psi$ is the compression to the first two rows and columns of the usual (irreducible)  representation of $\bH$ in $M_4(\bR)$.  So it is not maximal by Proposition \ref{DK617} (2).  It is interesting that 
 $\psi$ is pure, by Proposition \ref{DK617} (1). 
 (Cf.\ \cite[Example 3.27]{RSS}.)

 The latter also answers the question: For a real ucp map $\varphi$, there exists a maximal dilation of its complexification. Is the real part of this maximal dilation a maximal dilation of $\varphi$?  In the example above, the complexification of $\psi$ 
 has maximal dilation $x \mapsto x \oplus x$. The real part of this is a real dilation of $\psi$.   
 It is the map $\psi \oplus \psi$.  However  this is not maximal, since a summand of a maximal is easily seen to be maximal, and $\psi$ is not maximal.   
 This also follows by Proposition \ref{DK617} (2). 
 
3)\  It is easy to see that the converse of the first line in the remark fails even for $C^*$-algebras, answering the open question on p.\ 91 of \cite{Li}: if $\varphi$ is a pure state on $V_c$ then Re$ \, \varphi_{|V}$ on $V$ need not be pure.  (Consider $u^* \cdot u$ on $M_2(\bC)$.)  
 In the language above,   $x + iy$ pure in $K_c$ (which follows if
  $c(x,y)$ pure in $K_2$) need not imply $x$ pure in $K$.  
 Nonetheless, the class of such states Re$ \, \varphi_{|V}$ for pure $\varphi$ is an important class of states/an important subset of $K$.

    \bigskip
    
             For $K$ a complex compact nc convex set, one may ask about the relationships between the real and the complex pure/extreme/maximal  points of $K$. 
    We have:   
    
    \begin{lemma} \label{pekkcr}  For $K$ a complex compact nc convex set, if  $x \in K$ is complex pure (resp.\ nc extreme, maximal) in $K$ then $x$  is real pure (resp.\ nc extreme, maximal) in $K$.   
    Conversely, if $x$ is  real nc extreme and complex-irreducible  then $x$  is complex nc extreme and irreducible in $K$.   In particular, in general 
    $$\partial_{\bC} K \subsetneq \partial_{\bR} K.$$ 
    \end{lemma} 
    
    \begin{proof}  This is almost identical to Lemma \ref{pekkc}.    Indeed suppose that $x$  is complex  pure in $K$ and $x =  \sum_k \, \alpha_k^\tran x_k \alpha_k \in K_n$ for finite families $\{x_i \in K_{n_i}\}$ and $\{\alpha_i \in M_{n_i,n}(\bR)\}$, where $\sum \alpha_k^\tran \alpha_k = 1_n$.  
                Since $x$  is complex pure we have that a positive scalar multiple of $\alpha_k$ is  complex isometry matrix 
     $\beta_k \in M_{n_k,n}(\bC)$, 
     with $\beta_k^\tran x_k \beta_k = x$.  Clearly $\beta_k \in M_{n_k,n}(\bR)$.  So we have verified that $x$  is real 
     pure in $K$.      
          
     If $x$ is extreme in $K$ then in addition each 
     $x_k$  commutes with $\alpha_k \, \alpha_k^\tran$ and  $x_k \, \beta_k \, \beta_k^\tran = \beta_k x \beta_k^\tran$, which 
     is clearly (real and complex) unitarily equivalent to  $x$.    
    
    For the next part we first note that we already showed in Lemma \ref{pekkc} and Corollary \ref{maxiff} that  if $x \in K$ is complex maximal
      (resp.\ extreme) in $K_c$  then $x$  is complex maximal (resp.\ extreme) in $K$.        Suppose that  $x$ is real nc extreme  in $K$ and complex-irreducible.  Then  $x$ is complex nc extreme  in $K_c$ by Corollary \ref{chekkc}.       Then $x$ is real maximal in $K$, and hence complex maximal by Theorem \ref{pem}.
     A variant of the argument in Corollary \ref{chekkc} shows that if $x$ is irreducible for $M_n(\bC)$ then it is complex extreme.  Finally we give an example where $\partial_{\bC} K \neq \partial_{\bR} K$ in the next remark. 
     \end{proof} 
    
    {\bf Remark.} By a slight variant of the argument for the last assertion of  Corollary \ref{chekkc},  $x \in K_1$ is real extreme in $K_1$ if and only if $x$ is complex  extreme in $K$.  
    However generally the converses in Lemma \ref{pekkcr} fail for pure  (resp.\ extreme).  E.g.\ if $x$ is real nc extreme in $K$ then $x$ need not be complex extreme in $K$.   This fails even in $K_2$.  To see this
    consider an  extreme $x \in K$ such that 
    $x$ is not complex nc extreme in $K_c$, such as in Examples  \ref{cise} or \ref{Cal} (or Lemma \ref{pekkcr}).   However by the Remark     above Example \ref{cise}, $x$ is real nc extreme in $K_c$. 
Similarly $x$ is real pure in $K_c$, but not complex pure there (if it were it would be complex extreme since it is also complex maximal).
     
        \end{section}

\begin{section}{Noncommutative convex functions} \label{ncf} 

We shall see that the results with very lengthy proofs in Davidson and Kennedy’s chapter on nc convex functions, follow in the real case  very quickly by complexification.  We  saw this phenomenon often in \cite{BMc} (but not so much 
earlier in the present paper). 
However to achieve this, and for its own intrinsic interest, we will need to develop some basic theory of the complexification of nc convex functions and their     multivalued counterpart.    

We recall that  $f \in M_n(B(K))$ is identified with 
a graded function $f : K \to M_n(\cM)$.
It is unitarily equivariant for unitaries $U \in M_n$: $$f(UxU^*) =  (U \otimes I_n)f(x) (U^* \otimes I_n).$$
As in \cite[Section 7.3]{DK}, $f$ `respects
direct sums', and  
$f$ is selfadjoint if $f(x) = f(x)^*$ for $x \in K$.
As in \cite[Lemma 7.4]{BMc} we have $f_c \in M_n(B(K_c))$.  We recall that the isomorphism $M_n(B(K_c))
\cong M_n(B(K))_c$ is in one direction 
$f \mapsto {\rm Re} \circ f_{|K} + i \, {\rm Im} \circ f_{|K}$.  The inverse is the `complexification' $\Psi$ of the map
$f \mapsto f_c(x+iy) = u^* f(c(x,y)) u$ on $B(K,M_n)$, as we will discuss in more detail above Proposition  \ref{okfc0}.  
As in \cite{DK} the epigraph of a selfadjoint bounded nc function $f  \in M_n(B(K))$ is
the subset Epi$(f) \subseteq \bigsqcup_m \, (K_m \otimes M_n(M_m))$  defined by
$${\rm Epi}_m(f) = \{(x,\alpha) \in K_m \times M_n(M_m) : x \in K_m , \alpha \geq f(x) \}.$$
Davidson and Kennedy define $f$ to be (nc)  {\em convex} if Epi$(f)$  is an nc convex set, and to be 
(nc) {\em lowersemicontinuous} (or `lsc') if $\Epi(f)$  is closed.

 \begin{lemma}\label{LemIsomEquiv}
     Let  $K$ a real compact nc convex set and $f \in M_n(B(K))$.  Then $f$ is nc convex  if and only if
     \begin{equation}\label{eqIsomEquiv}
         f(\alpha^* x \alpha) \leq (1_n \otimes \alpha^*) f(x) (1_n \otimes \alpha),
     \end{equation} for every $x \in K_m$ and every isometry $\alpha \in M_{m,l}(\bR)$. 
 \end{lemma}
 \begin{proof}  In the complex case this is Remark $7.2.2$ in \cite{DK} or \cite[Proposition 16.13.2]{Dav}.  We include a proof, partly to serve as a model for the later more complicated 
 multivalued nc function case.
     For $f \in M_n(B(K))$, we have that $f$ is graded, and so $\Epi(f)$ is graded. Also $f$  preserves direct sums, which results in $\Epi(f)$ being closed under direct sums. Indeed, let $(x_i,\beta_i) \in \Epi_{m_i}(f)$,   and let $\alpha_i \in M_{k,m_i}(\bR)$ be a family of isometries such that $\sum_i \alpha_i \alpha_i^\tran = I_k$. We need to check that $$\sum_i\alpha_i (x_i,\beta_i)\alpha_i^\tran = \big(\sum_i \alpha_i x_i \alpha_i^\tran, \sum_i (1_n \otimes \alpha_i) \beta_i (1_n \otimes \alpha_i^\tran) \big) \in \Epi_k(f).$$
     Clearly, $\sum_i \alpha_i x_i \alpha_i^\tran \in K_k$. Also, $f(x_i) \leq \beta_i$ and adjoining both sides by $(1_n \otimes \alpha_i)$ and summing over $i$ gives that 
     $$\sum_i (1_n \otimes \alpha_i) f(x_i) (1_n \otimes \alpha_i^\tran) \leq \sum_i (1_n \otimes \alpha_i) \beta_i (1_n \otimes \alpha_i^\tran),$$
     and because $f$  respects  direct sums we get 
     $$f(\sum_i \alpha_i x_i \alpha_i^\tran) \leq \sum_i (1_n \otimes \alpha_i) \beta_i (1_n \otimes \alpha_i^\tran)$$ as we wanted.

     If (\ref{eqIsomEquiv}) holds, $\Epi(f)$ will be closed under compressions by isometries. If $(x,\beta) \in \Epi_m(f)$ and $\alpha \in M_{m,k}$ is an isometry then as before $\alpha^\tran x \alpha \in K_{k}$ and using (\ref{eqIsomEquiv}) gives
     \begin{align*}
         (1 \otimes \alpha^\tran) \beta (1 \otimes \alpha) &\geq (1 \otimes \alpha^\tran) f(x) (1 \otimes \alpha)
         \\&\geq f(\alpha^\tran x \alpha)
     \end{align*}
     and so $\alpha^\tran(x, \beta) \alpha \in \Epi_k(f)$. Conversely, if $\Epi(f)$ is a nc convex set then for all $x \in K_m$ and isometries $\alpha \in M_{k,m}$ we have $(x,f(x)) \in \Epi_m(f)$ and therefore $\big(\alpha^\tran x \alpha, (1_n \otimes \alpha^\tran) f(x) (1 \otimes \alpha)\big) \in \Epi_k(f)$ which gives (\ref{eqIsomEquiv}).
     \end{proof}

     In particular a bounded affine nc function is nc convex.

\bigskip 

{\bf Remark.} 
  Similarly, the Hansen-Pedersen type application 
in \cite[Proposition 7.2.3]{DK} holds in the real case, showing that $f$ is real nc convex  if and only if it is classically convex at every `level'.

\bigskip 

We now examine the relationship between lowersemicontinuous convex functions in $M_n(B(K_c))$ and $M_n(B(K))_c$. From \cite{BMc}, there are natural maps between these two sets. Namely, we had
    \[\Gamma: M_n(B(K_c)) \to M_n(B(K))_c\]
    \[\omega \mapsto \re \, \omega_{|K} + i \im \omega_{|K} \]
    and its inverse
    \[\Psi: M_n(B(K))_c \to M_n(B(K_c)) : f+ig \mapsto f_c + i g_c,\]
    or alternatively,  
    \[\Psi(f+ig)(x+iy) = (1_n \otimes u_m^*) \big(f(c(x+iy)) + i g(c(x+iy))\big) (1_n \otimes u_m)\]
    for $f+ig \in M_n(B(K))_c$ and $x+iy \in (K_c)_m$. 

    \begin{prop} \label{okfc0} 
        Let $f \in M_n(B(K))_{\rm sa}$.  Then $f$ is real nc convex (resp.\ lowersemicontinuous) if and only if $f_c = \Psi(f+i0)$ is complex convex (resp.\ lowersemicontinuous) as an element of $M_n(B(K_c))$.
    \end{prop}

 \begin{proof}
        If $f$ is real convex, then by  Lemma  \ref{LemIsomEquiv} we have that for all $x \in K_m$ and isometries $\alpha \in M_{k,m}(\bR)$ then
        $$f(\alpha^\tran x \alpha) \leq (1_n \otimes \alpha^\tran) f(x) (1_n \otimes \alpha).$$
        Let $z \in (K_c)_m$ and $\beta \in M_{k,m}(\bC)$ be an isometry. Then,
        \begin{align*}
            f_c(\beta^* z \beta)
            &= (1_n \otimes u_{k}^*) f(c(\beta^* z \beta)) (1_n \otimes u_k)
            \\&= (1_n \otimes u_{k}^*) f(c(\beta)^\tran c(z) c(\beta)) (1_n \otimes u_k)
            \\&\leq (1_n \otimes u_{k}^*c(\beta)^\tran) f( c(z)) (1_n \otimes c(\beta) u_k )
            \\&= (1_n \otimes \beta^* u_{m}^*) f( c(z) )
            (1_n \otimes u_m \beta)
            \\&= (1_n \otimes \beta^*) f_c(z) (1_n \otimes \beta).
        \end{align*}
        Conversely, suppose $f_c$ is convex and let $x \in K$ and $\alpha$ be an isometry. Then (\ref{eqIsomEquiv}) holds because $f(\alpha^\tran x \alpha) = f_c(\alpha^\tran x \alpha + i0)$ and since (\ref{eqIsomEquiv}) holds for $f_c$.

        If $f$ is lowersemicontinuous and $(x_\lambda, \alpha_\lambda) \to (x, \alpha)$, with the net in $\Epi(f_c)$,  then $(c(x_\lambda), c(\alpha_\lambda)) \to (c(x), c(\alpha))$. For $(z,\beta) \in \Epi(f_c)$ we have 
        $\beta \geq f_c(z) = u^* f(c(z)) u$. Since $f$ is real affine, $f(c(z))$ is of the form $c(w)$ for some $w \in \cM(\bC)$ (as in the  paragraph before Theorem \ref{cmax}).  It is easy to see that this implies that $u^* f(c(z)) u = w$ so that        $c(u^* f(c(z)) u ) = f(c(z)).$
                 Taking $c$ of both sides we get $c(\beta) \geq c(u^* f(c(z)u) =  f(c(z))$.               That is, $(c(z), c(\beta)) \in \Epi(f)$. 
                So $(c(x_\lambda), c(\alpha_\lambda))$ is in $\Epi(f)$, and hence so is its limit   $(c(x), c(\alpha))$. 
        That is, $f(c(x)) \leq c(\alpha)$. Reversing the argument above, adjoining  $u$ gives that $(x,\alpha) \in \Epi(f_c)$. 
        So $\Epi(f_c)$ is closed and $f_c$ is lsc.
        Conversely, if $\Epi(f_c)$ is closed and $(x_\lambda, \alpha_\lambda) \to (x, \alpha)$ with $(x_\lambda, \alpha_\lambda) \in \Epi(f)$, then $(x_\lambda +i0, \alpha_\lambda+i0) \to (x+i0, \alpha+i0)$ in $\Epi(f_c)$.
        Thus $(x,\alpha) \in \Epi(f)$.  So $\Epi(f)$ is closed and $f$ is lsc.
         \end{proof}

\begin{lemma}\label{Lemcfgi}
  Suppose that $f + ig \in M_n(B(K))_c = M_n(B(K_c))$   for a real compact nc convex set
   $K$.  Then $f + ig$ is complex
      nc convex if and only if  $c(f,g)$ is real nc convex
      in $M_{2n}(B(K))$. If these hold then $f$ is real nc convex. 
 \end{lemma}
 \begin{proof}  If $x \in K_m$ and $\alpha \in M_{m,l}(\bR)$ is an isometry, 
  then we need to know that $$(I_{2n} \otimes \alpha^{\tran}) \, c(f,g)(x) \, (I_{2n} \otimes \alpha) \geq c(f,g)(\alpha^{\tran} x \alpha).$$
  However we have that 
  $$(f+ig)(\alpha^{\tran} x \alpha) \leq 
  (I_{n} \otimes \alpha^{\tran}) \, (f+ig)(x) \, (I_{n} \otimes \alpha). $$
  Applying $c$ to both sides, and using that 
  $(I_{n} \otimes \alpha^{\tran}) \, (f+ig)(x) \, (I_{n} \otimes \alpha)$ is  $(I_{n} \otimes \alpha^{\tran}) \, f(x) \, (I_{n} \otimes \alpha) + i (I_{n} \otimes \alpha^{\tran}) \, g(x) \, (I_{n} \otimes \alpha),$ gives
  $$c(f,g)(\alpha^{\tran} x \alpha) \leq c((I_{n} \otimes \alpha^{\tran}) \, f(x) \, (I_{n} \otimes \alpha) , (I_{n} \otimes \alpha^{\tran}) \, g(x) \, (I_{n} \otimes \alpha)), $$
which equals  $(I_{2n} \otimes \alpha^{\tran}) \, c(f,g)(x) \, (I_{2n} \otimes \alpha)$  as desired.  

The last assertion follows easily from this by inspection of the top left corner. 

Conversely, suppose that $c(f,g)$ is real nc convex
      in $M_{2n}(B(K))$.   
      If $x  + iy \in (K_m)_c$ and $\alpha \in M_{m,l}(\bC)$ is an isometry, 
  then we need to show that $$(I_{n} \otimes \alpha^*) \, (f+ig)(x+iy) \, (I_{n} \otimes \alpha) \geq (f+ig)(\alpha^* (x +iy) \alpha).$$
Now $(f+ig)(x+iy) = u^* (f(c(x,y)) + i g(c(x,y))) u$.   Since 
$u (I_{n} \otimes \alpha) = (I_{2n} \otimes \alpha)  u$, it suffices to show that
$$(I_{2n} \otimes \alpha^*) (f(c(x,y)) + i g(c(x,y))) (I_{2n} \otimes \alpha) \geq 
f(c(\alpha^* (x +iy) \alpha)) + i g(c(\alpha^* (x +iy) \alpha)).$$ The latter equals $$f(c(\alpha)^\tran c(x,y) c(\alpha)) + ig(c(\alpha)^\tran c(x,y) c(\alpha)) .$$
Taking $c$ of both sides of the last inequality, we need that  
$$(I_{2n} \otimes c(\alpha)^{\tran}) c(f,g)(c(x,y)) (I_{2n} \otimes c(\alpha)) \geq c(f,g)(c(\alpha)^\tran c(x,y) c(\alpha)),$$ 
which holds if  $c(f,g)$ is real nc convex.  \end{proof}

{\bf Remarks.} 1)\ The last result extends to convex functions an important part  of the $B(K_c) = B(K)_c$ identification (see \cite[Lemma 7.4]{BMc}). We can interpret the bounded convex nc functions as a subset of the bounded nc functions, and then via the latter identification the last result is a very useful fact about the `complexification' of convex functions.

\smallskip

2)\ The proof shows that $f + ig$ is  complex  
      nc convex if and only if it is real 
      nc convex.

\bigskip

Now \cite[Proposition 7.2.8]{DK} is provable by complexification:    

 \begin{prop} \label{dk728}
       If $K$ is a  real compact nc convex set and if $\mu, \nu : C(K) \to M_m(\bR)$ are ucp
with $\mu(f ) = \nu(f)$ for all $n< \infty$ and $f \in M_n(C(K))$
nc convex, then $\mu =\nu$.  
    \end{prop}

 \begin{proof}
Indeed $\mu_c, \nu_c : C(K_c) \to M_m(\bC)$ are ucp.
Claim: $\mu_c(f + ig ) = \nu_c(f+ig)$ for all $n< \infty$ and $f + ig \in M_n(C(K)) = M_n(C(K_c))$ nc convex.   If the Claim is true then by the complex case (Proposition 7.2.8 in \cite{DK}), $\mu_c =\nu_c$ so that $\mu =\nu$ as desired.
Applying the map $c$, we see that $\mu_c(f + ig ) = \nu_c(f+ig)$
if $\mu_2(c(f,g)) = \nu_2(c(f,g))$, and the latter holds by hypothesis since $c(f,g)$ is nc convex by the previous lemma. \end{proof}

\section{Convex multivalued functions and convex  envelopes} \label{multi}

We now extend our real versions of convexity and lower semicontinuity to bounded multivalued selfadjoint nc
functions $F$.  We define the latter in the real case as in \cite[Definition 7.3.1]{DK} except of course matrices  have real entries (including   unitary matrices, naturally). Thus e.g.\ $F$ is bounded if there is a constant $\lambda > 0$  such that for
every $\beta \in F(x)$ there is an $\alpha \in F(x)$ with $\alpha \leq \beta$ and $\| \alpha \|
\leq \lambda$.  The graph  ${\rm Graph}(F)$ is defined by
 $${\rm Graph}_m(F) = \{ (x , \alpha ) : x \in K_m, \alpha \in F(x) \}.$$ Convexity and lower semicontinuity 
of $F$ are defined just as in \cite[Definition 7.3.2]{DK}, except all
matrices  have real entries.  That is, $F$ is convex (resp.\ lowersemicontinuous) if ${\rm Graph}(F)$ is nc convex (resp.\ closed). We say that $F \leq G$ if $G(x) \subseteq F(x)$ for all $x \in K$. 
Every real  (bounded) nc function $f : K \to M_n(\cM(\bR))$ gives a (bounded) multivalued selfadjoint nc
function $F(x) = [f(x),\infty)$, with Epi $(f)
= {\rm Graph}(F)$.  So $f$ is convex if and only if $F$ is convex.

  \begin{lemma}\label{bmext} For a real  nc convex set $K$, every real  bounded nc  multivalued selfadjoint function $F : K \to M_m(\cM(\bR))$ has a unique complex  bounded nc multivalued  extension $F_c : K_c \to M_m(\cM(\bC))$, and this is selfadjoint. 
           Similarly every real  bounded nc  selfadjoint  multivalued function $F : K \to M_m(\cM(\bC))$ has a unique complex  bounded selfadjoint nc multivalued  extension $F_c : K_c \to M_m(\cM(\bC))$.   Finally, $F \mapsto F_c$ is an order embedding:
           $F \leq G$ if and only if $F_c \leq G_c$.
            \end{lemma}

            \begin{proof}  For $x+iy \in (K_c)_n$ define $F_c(x + iy) = u^*  F(c(x+iy)) u$.   E.g.\ if $n = 1$ then 
            \[F_c(x + iy) = \frac{1}{\sqrt{2}} \begin{bmatrix}
                    1_n & i \cdot 1_n
                \end{bmatrix}
                \, F(c(x+iy)) \, 
                \frac{1}{\sqrt{2}} \begin{bmatrix}
                1_n \\ -i \cdot 1_n
            \end{bmatrix}\] 
            This function is a complex bounded nc selfadjoint multivalued function by a proof similar to \cite[Lemma 7.3]{BMc}. First, it is clearly graded and selfadjoint. Then, for $\beta \in M_n(\bC)$ a unitary and $x+iy \in (K_c)_n$ we have
            \begin{align*}
                F_c(\beta(x+iy)\beta^*) &= u^*                \, F(c(\beta)c(x+iy)c(\beta)^{\tran}) \, u
                   \\&= 
                         u^*  \, c(\beta)F(c(x+iy))c(\beta)^{\tran} \, u
            \\&= \beta u^* 
                \, F(c(x+iy)) \, u
    \beta^{*}
            \\&= \beta F_c(x+iy)\beta^{*}
            \end{align*}
            which shows the unitary invariance 
            of the extension. A similar proof shows the direct sum property.   To see that $F_c(x+iy) = F_c(x+iy) + M_m(\cM(\bC))^+$  choose $\alpha \in M_m(\cM(\bC))^+$.
            Then $\alpha = u^* c(\alpha) u.$ 
            For $z \in F(c(x+iy))$ 
            we know that $z + c(\alpha)
            \in F(c(x+iy))$.  Hence 
            $$u^* z u + u^* c(\alpha) u =u^* z u + 
            \alpha \in F_c(x+iy) .$$
             Thus $F_c(x+iy) = F_c(x+iy) + M_m(\cM(\bC))^+$. 
            
        We show that if $F$ is bounded then  $F_c$ is too.  Indeed let $\lambda > 0$ be  such that for
every $\beta \in F(x)$ there is an $\alpha \in F(x)$ with $\alpha \leq \beta$ and $\| \alpha \| 
\leq \lambda$.   If $\beta \in F_c(x + iy)$
then there exists $\gamma \in F(c(x + iy))$
with $\beta = u^* \gamma u$.   Choose $\alpha \in F(c(x + iy))$ with $\alpha \leq \gamma$ and $\| \alpha \| 
\leq \lambda$.  Then 
$u^* \alpha u \in  F(c(x + iy))$
with $u^* \alpha u \leq \beta$ and
$\| u^* \alpha u \| \leq \| \alpha \|
\leq \lambda$.   

            This will be the unique extension because for any complex nc function on $K_c$ extending $F$, say $G$, we have that
            \[(x+iy) \oplus (x-iy) = \frac{1}{2} \begin{bmatrix}
                1_n & i \cdot 1_n \\
                i \cdot 1_n & 1_n
            \end{bmatrix}c(x+iy) \begin{bmatrix}
                1_n & -i \cdot 1_n \\ -i \cdot 1_n & 1_n
            \end{bmatrix} . \] 
            So 
           $$    G(x+iy) \oplus G(x-iy)
            =G\Big(\frac{1}{2} \begin{bmatrix}
                1_n & i \cdot 1_n \\
                i \cdot 1_n & 1_n
            \end{bmatrix}c(x+iy) \begin{bmatrix}
                1_n & -i \cdot 1_n \\ -i \cdot 1_n & 1_n
            \end{bmatrix}\Big) ,$$ which equals               $$ \frac{1}{2} (I_m \otimes \begin{bmatrix}
                1_n & i \cdot 1_n \\
                i \cdot 1_n & 1_n
            \end{bmatrix} ) F(c(x+iy)) (I_m \otimes  \begin{bmatrix}
                1_n & -i \cdot 1_n \\ -i \cdot 1_n & 1_n
            \end{bmatrix} ).$$             Here we first used that $G(x \oplus y) = G(x) \oplus G(y)$. The last 
                        equation is a $2n \times 2n$ matrix with top-right and bottom-left corners being $0$. 
                        Comparing the top-left corners of the matrices in the above equation gives
             $G(x+iy) =  u^*  \, F(c(x+iy)) \, u$ 
          as desired.   

      Suppose that $F \leq G$ and that 
   $s + i t \in G_c(x+iy)$.
   Applying the map $c$ we see that 
   $$c(s,t) \in c(G_c(x+iy)) = G(c(x,y))
   \subset F(c(x,y)).$$  Reversing the argument 
   we see that $s + i t \in F_c(x+iy)$.  So 
   $F_c \leq G_c$.  The converse is easier, 
    since 
      $F$ and $G$ are `restrictions' of $F_c$ and $G_c$. So 
   $F \mapsto F_c$ is an order embedding.
          
          The second assertion is similar. \end{proof}

As in Lemma \ref{LemIsomEquiv}, we need a good test similar to (\ref{eqIsomEquiv}) for when nc  multivalued selfadjoint $F : K \to M_n(\cM(\bR))$ is convex.   It is simply that 
         \begin{equation}\label{eqnccEquiv}
         F(\alpha^* x \alpha) = (1_n \otimes \alpha^*) F(x) (1_n \otimes \alpha),
     \end{equation} 
         for all $x \in K_m$ and isometries $\alpha \in M_{k,m}(\bR)$. 
         Indeed clearly Graph$(F)$ is graded. Suppose that $(x,\beta) \in {\rm Graph}_m(F)$ and $\alpha \in M_{m,k}$ is an isometry.  We have $\beta \in F(x)$ and $\alpha^\tran x \alpha \in K_{k}$.  Using (\ref{eqnccEquiv}) gives
     \begin{align*}
         (1 \otimes \alpha^\tran) \beta (1 \otimes \alpha) & = (1 \otimes \alpha^\tran) F(x) (1 \otimes \alpha)
         \\&=  F(\alpha^\tran x \alpha)
     \end{align*}
     and so $\alpha^\tran(x, \beta) \alpha \in 
     {\rm Graph}(F)$. Conversely, if ${\rm Graph}(F)$ is a nc convex set then for all $x \in K_m$ and isometries $\alpha \in M_{k,m}$ we have $(x,F(x)) \in {\rm Graph}_m(F)$ and therefore $\big(\alpha^\tran x \alpha, (1_n \otimes \alpha^\tran) F(x) (1 \otimes \alpha)\big) \in {\rm Graph}(F)$ which gives (\ref{eqnccEquiv}).
         A similar argument to that in Lemma \ref{LemIsomEquiv} shows that Graph$(F)$ is closed under direct sums.  So $F$ is nc convex.

          For $G : K_c \to M_n(\cM(\bC))$ 
          a complex  bounded nc  multivalued selfadjoint function define $\pi(G) = {\rm Re} \, G_{|K}$.  The following shows that $\pi$ is a left inverse to the map $F \mapsto F_c$ in Lemma \ref{bmext}.

           \begin{lemma}\label{bmext2} Let $K$ be a real  nc convex set.
           If $G : K_c \to M_n(\cM(\bC))$ is a complex  bounded nc  multivalued selfadjoint function, then ${\rm Re} \, G_{|K} : K \to M_n(\cM(\bR))$ is a  real  bounded nc multivalued  selfadjoint function.  Moreover, $\pi(F_c) = F$
           if $F : K \to M_n(\cM(\bR))$ is 
           a  real  bounded nc multivalued  selfadjoint function, and $\pi(G_1) \leq \pi(G_2)$ if $G_1 \leq G_2$ are 
           complex  bounded nc  multivalued selfadjoint functions on $K_c$. 
            \end{lemma}

            \begin{proof} 
            Let $G: K_c \to M_n(\cM(\bC))$ be as above. If $k \in K$ then $$\pi(G)(k)^\tran = \re(G(k))^\tran = \re\big(G(k)^*\big) = \pi(G)(k).$$ So  $\pi(G)$ is  selfadjoint, and it clearly is nondegenerate and graded. For $\beta \in M_m(\bR)$ unitary and $x \in K_m$ we have that
            $$G(\beta (x+i0) \beta^*) = (1_n \otimes \beta) G(x+i0) (1_n \otimes \beta^*),$$
            so taking the real part of both sides, and using that $\re$ is unitarily equivariant (see Lemma 2.5 in \cite{BMc}), we get that $\pi(G)$ is unitarily equivariant. For direct sums, take $x,y \in K$ and then $G(x) \oplus G(y) = G(x \oplus y)$. Also $\re$ respects direct sums, so taking $\re$ of both sides gives that $\pi(G)$ respects direct sums.
            
            Fix $x \in K_m$ and $\alpha \in M_n(M_m(\bR))^+$.  If $z+ \alpha \in \pi(G)(x) + M_n(M_m(\bR))^+$ for $z = {\rm Re}(y)$ where $y \in G(x)$ and $\alpha \in
            M_n(M_m(\bR))^+$, then $y + \alpha \in G(x)$ and so $\re(y) + \alpha \in \pi(G)(x)$.  Thus $\pi(G)$ is upward directed. 
To show boundedness, let $x \in K$ and $\beta \in \pi(G)(x)$ so that $\beta = \re \, y$ for some $y \in G(x+i0)$. Because $G$ is bounded there exists a $\lambda > 0$ and $\alpha \in G(x+i0)$ with $\alpha \leq y$ and $\Vert \alpha \Vert \leq \lambda$. Taking the $\re$ of the inequality and using that $\re$ is a contraction gives $\re \, \alpha \leq \re \, y = \beta$ and $\Vert \re \, \alpha
\Vert \leq \lambda$. Thus, $\lambda$ is a bound for the function $\pi(G)$. 

That $\pi(F_c) = F$ for $F$ a real multivalued function is simple:
$$\pi(F_c)(x) = 
\re(F(x)) = F(x), \qquad x \in K.$$  
For the final statement of the lemma, suppose that $G_1 \leq G_2$ are complex nc multivalued functions so that $G_2(z) \subseteq G_1(z)$ for all $z \in K_c$. Then, for any $x \in K$ and $a \in \pi(G_2)(x)$ we have that $a = \re \, y$ for some $y \in G_2(x+i0) \subseteq G_1(x+i0)$. Therefore, $a \in \re(G_1(x+i0)) = \pi(G)(x)$. So $\pi(G_1) \leq \pi(G_2)$.
            \end{proof}

            In the sequel we continue to write $\pi(G)$ for 
            ${\rm Re} \, G_{|K}$.

           \begin{lemma}\label{bmend} Let $K$ be a real  nc convex set, and 
            $G : K_c \to M_n(\cM(\bC))$ a complex  bounded nc  multivalued selfadjoint function.
            If $G$ is nc convex (resp.\ lowersemicontinuous) then 
            ${\rm Re} \, G_{|K} : K \to M_n(\cM(\bR))$ is real nc convex (resp.\ lowersemicontinuous). 
            \end{lemma}

            \begin{proof} 
             By the test in (\ref{eqnccEquiv}) it is easy to see that $\pi(G)$ is real nc convex if 
            $G$ is complex nc convex.               
             Suppose that $G$ is 
                           lowersemicontinuous and $(x_\lambda, \alpha_\lambda) \to (x,\alpha)$ for $(x_\lambda, \alpha_\lambda) \in {\rm Graph}(\pi(G))$. Because $\alpha_\lambda \in \re \, G(x_\lambda+i0)$ we can add to both sides a representative $i s_\lambda$ of $i \im(G(x_\lambda+i0))$,  for $s_\lambda \in \im G(x_\lambda)$.  We  get $\alpha_\lambda + i s_\lambda \in G(x_\lambda+i0)$.  We may assume without loss that $(s_\lambda)$ is uniformly bounded.  Indeed since $G$ is bounded there is a uniformly bounded
    $\beta_\lambda + i t_\lambda \leq
    \alpha_\lambda + i s_\lambda$, with $\beta_\lambda + i t_\lambda \in G(x_\lambda)$.
    Then $$\alpha_\lambda + i t_\lambda  
    = (\alpha_\lambda - \beta_\lambda) + \beta_\lambda + i t_\lambda \in G(x_\lambda),$$ since 
    $G(x_\lambda)$ is upwards directed.   Also 
    $(\alpha_\lambda + i t_\lambda)$ and $(t_\lambda)$ are bounded. 
    
    We have $(x_\lambda, \alpha_\lambda + i s_\lambda) \in {\rm Graph}(G)$.   Since  $(s_\lambda)$ is a bounded net it has a weak* convergent subnet, call it $(s_{\lambda_\mu})$ with limit  $s$. So, $(x_{\lambda_\mu}, \alpha_{\lambda_\mu} + i s_{\lambda_\mu}) \to (x, \alpha + i s)$.
    Hence  $(x, \alpha + is) \in {\rm Graph}(G(x+i0))$ because $G$ is lsc. Taking the real part of $\alpha + is \in G(x+i0)$ gives $\alpha \in \pi(G)(x+i0)$.  Thus  $(x, \alpha)$ is in the graph of $\pi(G)$.   So this graph is closed 
    and $\pi(G)$ is lsc. \end{proof}

 \begin{prop} \label{okfc} For a real  nc convex set $K$, and a real  bounded nc  multivalued selfadjoint function $F : K \to M_n(\cM(\bR))$, we have that 
         $F$ is real convex (resp.\ lowersemicontinuous) if and only if $F_c$ is complex convex (resp.\ lowersemicontinuous).
    \end{prop}

 \begin{proof} 
        If $F$ is real convex, then by (\ref{eqnccEquiv}) we have 
        $$F(\alpha^\tran x \alpha) = (1_n \otimes \alpha^\tran) F(x) (1_n \otimes \alpha)$$
        for all $x \in K_m$ and isometries $\alpha \in M_{k,m}(\bR)$.
        Let $z \in (K_c)_m$ and $\beta \in M_{k,m}(\bC)$ be an isometry. Then,
        \begin{align*}
            F_c(\beta^* z \beta)
            &= (1_n \otimes u_{k}^*) F(c(\beta^* z \beta)) (1_n \otimes u_k)
            \\&= (1_n \otimes u_{k}^*) F(c(\beta)^\tran c(z) c(\beta)) (1_n \otimes u_k)
            \\& =  (1_n \otimes u_{k}^*c(\beta)^\tran) F( c(z)) (1_n \otimes c(\beta) u_k )
            \\&= (1_n \otimes \beta^* u_{m}^*) F( c(z) )
            (1_n \otimes u_m \beta)
            \\&= (1_n \otimes \beta^*) F_c(z) (1_n \otimes \beta).
        \end{align*}
        The converse follows from the last lemma.
        Alternatively, suppose $F_c$ is convex and let $x \in K$ and $\alpha$ be an isometry. Then (\ref{eqnccEquiv}) holds for $F$ because $F(\alpha^\tran x \alpha) = F_c(\alpha^\tran x \alpha)$ and since (\ref{eqnccEquiv}) holds for $F_c$.

        The remainder of the proof is a modification of arguments for  Proposition \ref{okfc0}.  Suppose that  $F$ is lowersemicontinuous and $(z_\lambda, \beta_\lambda) \to (z, \beta)$, with the net in ${\rm Graph}(F_c)$. Then $(c(z_\lambda), c(\beta_\lambda)) \to (c(z), c(\beta))$.  By a `multivalued variant' of an argument in the proof 
        of Proposition \ref{okfc0} we have that if $(z, \beta) \in {\rm Graph}(F_c)$ then since $F(c(z))$ consists of elements of the form $c(w)$ for some $w$, we have 
        $c(u^* F(c(z)) u)  =  F(c(z))$ and  $(c(z), c(\beta)) \in {\rm Graph}(F)$. 
        Since ${\rm Graph}(F)$ is closed and $(c(z_\lambda), c(\beta_\lambda)) \in {\rm Graph}(F)$  we get $(c(z), c(\beta)) \in {\rm Graph}(F)$.  From this, a compression by $1 \otimes u$ gives that $(z, \beta) \in {\rm Graph}(F_c)$.  So the latter graph is closed and $F_c$ is lsc.   The converse is easier and similar, but following the converse in the last lines of the proof 
        of Proposition \ref{okfc0}. Indeed it also follows from the last lemma. \end{proof}

 The rest of \cite[Section 7.3]{DK} is the same in the real case.  Turning to \cite[Section 7.4]{DK}
 we define the {\em real convex
envelope} $\bar{F}$ of a real bounded selfadjoint multivalued nc function $F : K \to M_n(\cM(\bR))$ to be the
multivalued nc function $\bar{F} : K \to M_n(\cM(\bR))$ with 
Graph $(\bar{F})= \overline{{\rm ncconv}}({\rm Graph}(F)).$ It is easy to see that \cite[Proposition 7.4.2]{DK} works the same in the real case, and in particular:

\begin{prop} $\bar{F}$ is the largest convex nc
selfadjoint multivalued function smaller than $F$.
\end{prop}

     {\bf Remark.}  By splitting into the real and imaginary matrix parts, every complex real selfadjoint bounded multivalued nc function on $K_c$ is an $F + iG$.   We may also define $c(F,G)$ in the obvious way.  It is probably true by an adjustment of the proof of 
     Lemma \ref{Lemcfgi} that a complex bounded nc  multivalued selfadjoint function $F + iG$ on $K_c$ is
   complex  convex (resp.\ lowersemicontinuous) if and only if $c(F,G)$ is.  We have not taken the time to check this.  

\begin{lemma}\label{barre} Let $K$ be a real  nc convex set, and $F : K \to M_n(\cM(\bR))$  a  real  bounded nc multivalued  selfadjoint function. 
Then ${\rm Re} \, \overline{F_c}(x) = \bar{F}(x)$ 
for $x \in K$.   Moreover, $(\bar{F})_c = \overline{F_c}$.
            \end{lemma}

            \begin{proof} 
For $x \in K$ we have 
$\overline{F_c}(x) = \cap_{{\rm convex} \, G \leq F_c} G(x),$ the intersection over complex bounded selfadjoint multivalued nc functions $G \leq F_c$.
Thus
$${\rm Re} \, \overline{F_c}(x) = \cap_{{\rm convex} \, G \leq F_c} \; {\rm Re} \, G(x)  = \cap_{{\rm convex} \, G \leq F_c} \, \pi(G)(x),$$
the intersection over such $G$.  However the set of such $\pi(G)$ is exactly the collection of 
real  bounded selfadjoint multivalued nc functions $H \leq F$ (using that $\pi$ is decreasing, $\pi(F_c) = F$, and that $F \mapsto F_c$ is an order embedding).  Thus
$${\rm Re} \, \overline{F_c}(x) = \cap_{{\rm convex} \, H \leq F} H(x) = \bar{F}(x)$$
as desired. 

To see that $(\bar{F})_c = \overline{F_c}$ it suffices by the uniqueness in Lemma \ref{bmext} 
to show that $\overline{F_c}(x) = \bar{F}(x)$ for $x \in K$.  By the last paragraph it suffices  
to show that Im $\overline{F_c}(x) = 0$.
Choose a nc convex $G \leq F$ (e.g.\ $G = \bar{F}$).
Then $G_c \leq F_c$ and $${\rm Im} \, 
\overline{F_c}(x) \subset {\rm Im} \, 
G_c(x) = {\rm Im} \, 
G(x) = 0,$$ as desired. 
\end{proof}

Davidson and Kennedy's Theorem 7.4.3 is a noncommutative analogue of the highly important classical fact
that the convex envelope of a function is the supremum of
the continuous affine functions below the function.
To explain one  notation here: for a multivalued nc function $F : K \to M_n(\bR)$  we define a multivalued function $I_m \otimes F : K \to M_m(M_n(\bR))$  by
$$(I_m \otimes F)(x) = \{ I_m \otimes \alpha : 
\alpha \in F (x) \}, \qquad x \in K.$$  
They point out that this is not a nc multivalued function in the technical sense.  

\begin{thm} \label{dklong} Let $K$ be a real compact nc convex set
      and $F : K \to M_n(\cM(\bR))$ a real selfadjoint bounded multivalued nc function.  Then  $$\bar{F}(x) = \cap_{m \in \bN} \cap_{a 
      \leq I_m \otimes F} \; \{ \alpha \in (M_n(M_p(\bR))_{\rm sa} : a(x) \leq I_m \otimes \alpha \}, \qquad x \in K_p,$$
   where the intersection is taken over all selfadjoint affine nc
functions $a \in M_m(M_n(A(K)))_{\rm sa}$  satisfying $a 
      \leq I_m \otimes F$. The same
holds if we intersect over all $m \leq \kappa$. \end{thm}

 \begin{proof}  For $x \in K_p$ 
 define $\tilde{F}(x)$ to be the intersection in the statement of the theorem.
 We have by \cite[Theorem 7.4.3]{DK} that 
 $$\overline{F_c}(x) = \cap_{m \in \bN} \cap_{a 
      \leq I_m \otimes F_c} \; \{ \alpha \in (M_n(M_p(\bC))_{\rm sa} : a(x) \leq I_m \otimes \alpha \},$$
      the intersection  taken over  all complex selfadjoint affine nc
functions $a \in M_m(M_n(A_{\bC}(K_c)))_{\rm sa}$  satisfying $a 
      \leq I_m \otimes F_c$. 
      Taking real parts and using the lemma,
      $$\bar{F}(x)  = {\rm Re} \; \overline{F_c}(x) = \cap_{m \in \bN} \cap_{a 
      \leq I_m \otimes F_c} \; \{ {\rm Re} \; \alpha \in (M_n(M_p(\bR))_{\rm sa} : a(x) \leq I_m \otimes \alpha \}, $$ the intersection  over all complex  affine nc
functions $a$ as above on $K_c$.  
      We show that the latter set is contained in 
$\tilde{F}(x)$. 
       Indeed if $a = a^* \in M_m(M_n(A(K)))_{\rm sa}$ 
      with $a 
      \leq I_m \otimes F$, then 
      $a_c$ is a complex selfadjoint affine nc
function   satisfying $a_c
      \leq I_m \otimes F_c$, using the fact that 
      $F \mapsto F_c$ is an order embedding.
       So the intersection in the definition of $\tilde{F}(x)$  is over fewer sets, so is bigger. 
             
      Conversely, suppose that $x \in K_p$ and
      $\alpha \in \tilde{F}(x)$, and $b$ is a complex selfadjoint affine nc
function in $M_m(M_n(A_{\bC}(K_c)))_{\rm sa}$  satisfying $b 
      \leq I_m \otimes F_c$.   Thus,
      $b(x) \leq I_m \otimes \beta$ for all $\beta \in F_c(x) = F(x)$.
       Applying the $c$ map, which is an order embedding, we have 
       $$c(b(x)) = c(b)(x) \leq I_2 \otimes (I_m \otimes \beta) = I_{2m} \otimes \beta .$$
       That is, $c(b(x))  \leq I_{2m} \otimes F(x)$.     We thus have  by hypothesis that
    $$c(b(x))     \leq I_{2m} \otimes \alpha = c(I_{m} \otimes \alpha). 
    $$
Since $c$ is an order embedding, $b(x) \leq I_{m} \otimes \alpha$.
Thus $\alpha$ is in the intersection above characterizing $\bar{F}(x)$, as desired.    \end{proof}
 
It seems that \cite[Corollary 7.4.4]{DK} works the same in the real case.  Finally (for Section 7.4), the noncommutative analogue of a result of Mokobodzki works with the same proof:

\begin{cor} Let $K$ be a real compact nc convex set
      and $F : K \to M_n(\cM(\bR))$ a real selfadjoint bounded multivalued nc function
      with
convex envelope $\bar{F}$.
      Then for $x \in K_p$, 
      $$\overline{F}(x) = \bigcap_{g \leq 1_m \otimes F} \{\alpha \in (M_n(M_p(\bR)))_{\rm sa}: 1_m \otimes \alpha \geq g(x)\},$$ where the intersection is taken over all $m$ and all selfadjoint convex nc
functions $g \in M_m(M_n(C(K)))_{\rm sa}$  satisfying $g 
      \leq I_m \otimes F$. 
    \end{cor}
    
 Turning to \cite[Section 7.5]{DK}, this contains two remarkable results.   The first says that the `convex envelope encodes information about the set of representing maps of a point' \cite{DK}. Again the complex proof works. 

 \begin{thm} Let $K$ be a real compact nc convex set and 
 $f : K \to M_n(\cM(\bR))$ a selfadjoint lowersemicontinuous bounded nc function
with convex envelope $\bar{f}$. Then for $x \in K_m$, 
$$\overline{f}(x) = \bigcup_\mu [\mu(f),+\infty)$$
where the union is taken over all ucp $\mu : C(K) \to M_m(\bR)$ with barycenter $x$. 
    \end{thm}

      \begin{cor} \label{cor753} Let $F : K \to M_n(\cM(\bR))$ be a real  bounded nc  multivalued selfadjoint function on  a real  nc convex set $K$.
      For
every unital completely positive map $\mu : C(K) \to M_p(\bR)$ 
we have that $$\mu(\bar{F}) = \cap_{m \in \bN} \cap_{g 
      \leq I_m \otimes F} \; \{ \alpha \in (M_n(M_p(\bR))_{\rm sa} : \mu(g) \leq I_m \otimes \alpha \},$$ 
where the intersection is taken over all m and all convex nc functions $g \in M_m(M_n(C(K)))$ with 
$g       \leq I_m \otimes F_c$.    \end{cor}

 \begin{proof}   Define $\tilde{\mu}$ to be the intersection in the statement of the theorem.
 We have by \cite[Theorem 7.5.3]{DK} that 
 $$\mu_c(\overline{F_c}) = \cap_{g 
      \leq I_m \otimes F_c} \; \{ \alpha \in (M_n(M_p(\bC))_{\rm sa} : \mu_c(g) \leq I_m \otimes \alpha \},$$
      the intersection  taken over  all complex selfadjoint convex nc
functions $g \in M_m(M_n(C_{\bC}(K_c)))_{\rm sa}$  satisfying $g
      \leq I_m \otimes F_c$. 
      Taking real parts and using Lemma \ref{barre}, we have
     $$\mu(\bar{F}) = \cap_{g 
      \leq I_m \otimes F_c} \; \{ {\rm Re} \, \alpha \in (M_n(M_p(\bR))_{\rm sa} : \mu_c(g) \leq I_m \otimes \alpha \},$$
      the intersection  taken over the same collection of complex selfadjoint convex nc
functions $g$ as above.  As in the proof of Theorem \ref{dklong} we have $\mu(\bar{F}) \leq \tilde{\mu}$, and on the other hand 
for $\alpha \in  \tilde{\mu}$ and for a complex selfadjoint convex nc
function $g \leq I_m \otimes F_c$ as above, we have 
$$c(g(x)) = c(g)(x) \leq (I_{2m} \otimes F)(x), \qquad x \in K.$$
We thus have by definition of $\tilde{\mu}$  that
$$c(\mu_c(g)) = \mu(c(g)) \leq I_{2m} \otimes \alpha = c(I_{m} \otimes \alpha).$$
As in Theorem \ref{dklong} this means that 
$\mu_c(g) \leq I_{m} \otimes \alpha$, so that
$\alpha \in \mu(\bar{F})$ as desired. 
\end{proof}  

   Just as in \cite{DK}, and with the same proof, this  implies a Jensen inequality:
   
\begin{thm} {\rm (Noncommutative Jensen inequality)}\ Let $K$ be a real compact nc convex set and $f \in B(K)$ a selfadjoint lowersemicontinuous convex nc function. For any completely positive map
$\mu : C(K) \to M_n(\bR)$  with barycenter $x
\in K_n$, we have $f(x) \leq \mu(f).$ 
    \end{thm}

\end{section}

 \begin{section}{Appendix by Travis 
 Russell: A more detailed proof of Theorem \ref{DK643}.} 
        \label{trav}

As 
said 
(together with a few more specifics) below Theorem \ref{DK643},  we will use below that  
various properties of irreducible representations and corresponding GNS representations from   \cite{Dix}, hold in the real case.  
    We amplify Davidson and Kennedy's argument, in part to see that there are no  hidden real vs.\ complex issues.  We will also use the following modified version of \cite[Proposition 2.2.9]{DK}: Let $K$ be a closed nc-convex set over a dual operator space $E$. Suppose there is a net of contractions $\{\alpha_i \in M_{n,n_i}\}$ and a net $\{x_i \in K_{n_i}\}$ such that $\lim \alpha_i \alpha_i^* = I_n$ and such that $\lim \alpha_i x_i \alpha_i^* = x \in M_n(E)$. Then $x \in K_n$.
To see this, let $\alpha_i = U_i P_i$ be a polar decomposition for $\alpha_i$, where $U_i \in M_{n,n_i}$ is a partial isometry and $P_i \in M_{n_i}$ is a positive operator. Let $\beta_i = U_i (I_{n_i} - P_i^2)^{1/2} \in M_{n,n_i}$. Then $\alpha_i \alpha_i^* + \beta_i \beta_i^* = U_i U_i^*$, which is a projection in $M_n$. Let $m_i \leq n$ be a cardinal and $V_i \in M_{n,m_i}$ be a partial isometry such that $V_i V_i^* + U_i U_i^* = I_n$. Then because $\alpha_i \alpha_i^* \to I_n$ in $M_n$, we have $\beta_i \beta_i^* \to 0$ and $V_i V_i^* \to 0$ in $M_n$. Fix $y \in K_1$. Since $K$ is closed under nc-convex  combinations, $\alpha_i x_i \alpha_i^* + \beta_i (y \otimes I_{n_i}) \beta_i^* + V_i (y \otimes I_{m_i}) V_i^* \in K_n$ for every $i$. Since e.g.\ $\beta_i  \beta^*_i \to 0$ SOT, we have $$\alpha_i x_i \alpha_i^* + \beta_i (y \otimes I_{n_i}) \beta_i^* + V_i (y \otimes I_{m_i}) V_i^* \to x$$ in $M_n(E)$.  So  $x \in K_n$ since $K$ is closed. 
    
    Let $Z = \overline{\text{ncconv}} \{ \delta_x : x \in X\} \subseteq L$, where $L$ is the nc state space of $C(K)$. Let $y \in K$. Then $y$ is approximated by an nc convex combination $\sum \alpha_i^* x_i \alpha_i$ with $x_i \in X$ and 
    $\sum \alpha_i^*  \alpha_i = I$. Recall that the barycenter map from $L$ to $K$ is given by restriction of an nc state on $C(K)$ to $A(K)$, and this map is nc-affine and continuous. Now the barycenter of $\sum \alpha_i^* \delta_{x_i} \alpha_i$ is $\sum \alpha_i^* x_i \alpha_i$ since the barycenter map is nc-affine. Since $Z$ is compact and the barycenter map is continuous, there exists $\mu \in Z$ such that $y = \mu |_{A(K)}$ (i.e.\  $y$ is the barycenter of $\mu$).

    Let $y \in \partial K$ and suppose $y$ is the barycenter of $\mu \in Z$. By Theorem 6.1.9., $\delta_y: C(K) \to B(H)$ is irreducible and the unique extension of $y$ to $C(K)$, hence $\mu = \delta_y$. Let $n \in \mathbb{N}$, and fix an $n$-dimensional subspace $H' \subseteq H$ and an orthonormal basis $\{e_1, \dots, e_n\}$ for $H'$. Consider $\pi = \delta_y \otimes id_n: M_n(C(K)) \to B(H \otimes \mathbb{R}^n)$. This map is irreducible.
    
    Define $\psi: M_n(C(K)) \to \mathbb{R}$ by \[ \psi( (a_{ij})) = \frac{1}{n} \sum_{i,j} \langle \delta_y(a_{ij}) e_j, e_i \rangle_{H'} = \langle (\pi(a)_{ij}) (\oplus_i \frac{1}{\sqrt{n}} e_i), (\oplus_i \frac{1}{\sqrt{n}} e_i ) \rangle_{H \otimes \mathbb{R}^n}. \]
    Since $h = (1/\sqrt{n}) (\oplus_i e_i)$ is a unit vector, and since $\pi = \delta_y \otimes id_n$ is irreducible, $h$ is a cyclic vector (e.g. see Theorem 5.1.5 (2) of \cite{Mu}). It follows that the GNS representation $\pi_{\psi}: M_n(C(K)) \to B(H_{\psi})$ corresponding to $\psi$ is unitarily equivalent to the (irreducible) representation $\delta_y \otimes id_n$ (e.g. see Theorem 5.1.4 of \cite{Mu}). Consequently  $\psi$ factors through $\delta_y \otimes id_n$ and $\psi$ is pure (since its corresponding representation is irreducible).

    Since $\delta_y \in Z$ and since $Z$ is the closed nc convex hull of $\{ \delta_x : x \in X\}$, if $a \in \cap_{x \in X} \ker(\delta_x)$, then $a \in \ker(\delta_y)$. So $\cap_{x \in X} \ker(\delta_x) \subseteq \ker(\delta_y)$. Consequently, $\cap_{x \in X} \ker(\delta_x \otimes id_n) \subseteq \ker(\delta_y \otimes id_n)$. Because the state $\psi$ factors through $\delta_y \otimes id_n$ and hence vanishes on $\ker(\delta_y \otimes id_n)$, the real case of Dixmier's     Proposition 3.4.2(ii) implies that $\psi$ is a weak-$*$ limit of pure states $\psi_{x_i}: M_n(C(K)) \to \mathbb{C}$ factoring through representations $\delta_{x_i} \otimes id_n$, i.e. $\psi_{x_i}(z) = \langle (\delta_{x_i} \otimes id)(z) h_i, h_i \rangle$ for some unit vector $h_i \in H_i^n$, where $\delta_{x_i}: C(K) \to B(H_i)$. 

    Define $\rho_{x_i}: C(K) \to M_n$ by $(\rho_{x_i}(x))_{k,l} := n\psi_{x_i}(x \otimes E_{k,l})$ and $\rho_y: C(K) \to M_n$ by $(\rho_y(z))_{kl} = n \psi(z \otimes E_{kl})$. Now $$\langle (\delta_y \otimes id) (z) e_l, e_k \rangle = n\psi(z \otimes E_{kl}) = \rho_y(z)_{kl}$$ for all $k,l$, so $\rho_y$ is the compression of $\delta_y$ to $H'$ (identifying $M_n = B(H')$ and $\mathbb{R}^n = H'$). Also,  $\rho_{x_i} \to \rho_y$, since $\psi_{x_i}(z \otimes E_{kl}) \to \psi(z \otimes e_{kl})$ for all $z \in C(K)$. 

    Write each $h_i \in H_i^n$ as $h_i = \oplus_{k=1}^n h_i^k$. Because $$\psi_{x_i}(z \otimes E_{kl}) = \langle (\delta_{x_i} \otimes id_n) (z \otimes E_{kl}) h_i, h_i \rangle = \langle \delta_{x_i} h_l, h_k \rangle,$$ we have $\rho_{x_i} = V_{x_i}^* \delta_{x_i} V_{x_i}$, where $V_{x_i}: H' \to H_i$ is defined by $V_{x_i} e_k = n^{1/2} h_i^k$  for all $k$. Because $\psi_{x_i}(I \otimes E_{kl}) \to \psi(I \otimes E_{kl}) = \delta_{kl}$, $\rho_{x_i}(I) \to I_{H'}$. Hence $V_{x_i}^* V_{x_i} \to I_{H'}$. Since $\{V_{x_i}^* V_{x_i}\}$ is a net on the finite-dimensional Hilbert space $H'$, $\|V_{x_i}^* V_{x_i}\| \to 1$ and therefore $\|V_{x_i}\| \to 1$. Thus, if we let $\alpha = \|V_{x_i}\|^{-1} V_{x_i}^*$, then $\alpha_{x_i} \delta_{x_i} \alpha_{x_i}^* = \|V_{x_i}\|^{-2} \rho_{x_i} \to \rho_y$ and $\alpha_{x_i} \alpha_{x_i}^* \to I_n$. By continuity of the barycenter map again, $\lim \alpha_{x_i} x_i \alpha_{x_i}^* \in X$. Therefore $X$ contains the image of every finite-dimensional compression of $y$.

    To finish the proof, assume $y \in K_m$ for some infinite cardinal $m$. By considering the finite-dimensional compressions of $y$, there exists a net of isometries $\beta_i \in M_{m,d_i}$ (with $d_i$ finite) and $y_i \in X_{d_i}$ (by the above arguments) such that $\beta_i y_i \beta_i^* \to y$ and $\beta_i \beta_i^* \to I_m$. Here we can take $y_i = \beta_i^* y \beta_i$. So $\beta_i \beta_i^* \to I_n$ weak*. By the closure of $X$ under limits of this form, we see that $y \in X$. 
\end{section}

\end{section}

   \subsection*{Acknowledgements.}   
  We acknowledge support from NSF Grant DMS-2154903.  We thank Matt Kennedy and Scott McCullough for very helpful discussions and historical  comments.  We also deeply thank Travis Russell for several discussions, and for the Appendix.

\end{document}